\documentclass[a4paper,11pt]{amsart}
\usepackage{amsfonts,amssymb,epsfig,latexsym}
\usepackage{amsmath}
\usepackage[latin1]{inputenc}
\usepackage{color,layout}
\usepackage{ifthen}

\usepackage{pdfsync}
\usepackage{graphicx}
\usepackage{color}

\newtheorem{theorem}{Theorem}[section]
\newtheorem{lemma}[theorem]{Lemma}
\newtheorem{proposition}[theorem]{Proposition}

\newtheorem{remark}[theorem]{Remark}

\def\square{\hbox{\vrule\vbox{\hrule\phantom{o}\hrule}\vrule}}

\def\R{\mathbb {R}}
\def\C{\mathbb {C}}
\def\N{\mathbb {N}}
\def\Z{\mathbb {Z}}

\def\tendsto{\rightarrow}
\def\re{\mathop{\rm Re}\nolimits}
\def\im{\mathop{\rm Im}\nolimits}
\def\la{\langle}
\def\ra{\rangle}
\def\O{\mathcal O}
\newcommand{\be}{\begin{equation}}
\newcommand{\ee}{\end{equation}}

\numberwithin{equation}{section}

\begin{document}

\title[Lower bound of resonances for Helmholtz resonator]{Optimal lower bound of the resonance widths for the Helmholtz Resonator}
 
\begin{abstract}
Under a geometric assumption on the region near the end of its neck, we prove an optimal exponential 
lower bound on the widths of resonances for a general two-dimensional Helmholtz resonator. An extension 
of the result  to the $n$-dimensional case, $n\leq 12$, is also obtained.
\end{abstract}

\author{Andr\'e Martinez  \& Laurence N\'ed\'elec}


\maketitle

\addtocounter{footnote}{1}

\footnotetext{A.M.: Universit\`a di Bologna, Dipartimento di Matematica, Piazza di Porta San Donato 5, 40127 Bologna,
Italy. Partly supported by Universit\`a di Bologna, Funds for Selected Research Topics and Founds for Agreements with
Foreign Universities
 
L.N.: Department of Mathematics, Stanford University, Stanford, California 94305}

\setcounter{section}{0}

\section{Introduction}

A resonator consists of a bounded cavity (the chamber) connected to the exterior by a thin tube (the neck of the chamber). 
The frequencies of the sounds it produces are determined by the shape of the chamber, while their duration by the length and the width of the neck  in a non-obvious way,
and our goal is to understand these. Mathematically, this phenomenon is described by the resonances of the
Dirichlet Laplacian $-\Delta_\Omega$ on the domain $\Omega$ consisting of the union of the chamber, 
the neck and the exterior (see Figure  \ref{fig1}).

 \begin{figure}[h]
 \label{fig1}

 \vspace*{0ex}
 \centering
 \includegraphics[width=0.8\textwidth,angle=0]{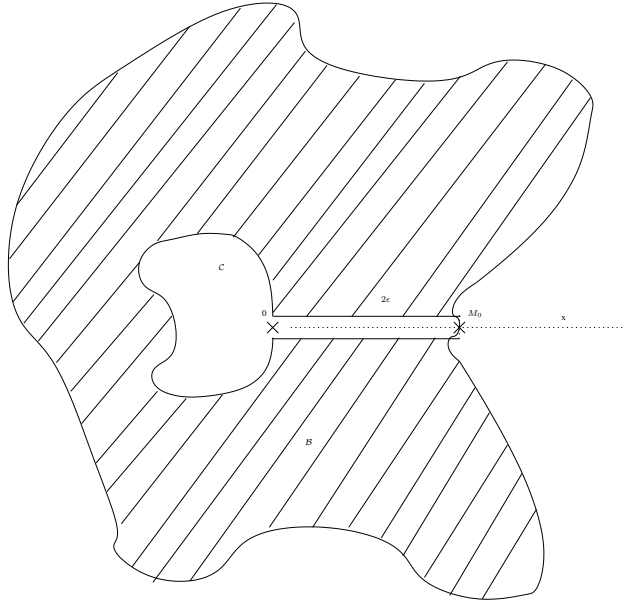}
 \vspace*{0ex}
 \caption{The Helmholtz resonator}
 \end{figure}

This article extends our previous work \cite{MN}, in that we are now able to handle regions where the shape of the exterior 
is quite general, although the shape of the neck stays the same. The main changes appear in sections \ref{largedisc},
\ref{eastob} and \ref{eastbound}, where Carleman estimates are used, and Green's identity is replaced by an estimate to 
obtain a lower bound on the imaginary part of the resonances. 

We recall that resonances are the eigenvalues of a complex deformation of $-\Delta_\Omega$;  their real and imaginary
parts are the frequencies and inverses of the half-lives, respectively, of the corresponding vibrational modes. It is of obvious physical 
interest to estimate these two quantities as precisely as  possible. One practical way to do this involves studying 
this problem in the asymptotic limit when the width $\varepsilon$ of the neck tends to zero. Those resonances whith imaginary parts 
tending to zero converge to the eigenvalues of the Dirichlet Laplacian on the cavity, and there is an exponentially small upper bound 
for the absolute values of the imaginary parts (the widths) of the resonances \cite{HM}. However, without very 
restrictive hypotheses, no lower bound is known. We mention in particular that lower bounds are known in
the one-dimensional case \cite{Ha, HaSi}. As for the higher dimensional case, we mention \cite{FL, Bu2, HS} 
which contain results concerning exponentially small widths of quantum resonances, but these do not apply to a Helmholtz resonator.
We also mention that the semiclassical lower bound obtained in \cite{HS} is optimal (see also \cite{FLM} for a generalization).

Here, we obtain an optimal lower bound (see Theorem \ref{mainth}) under a geometric condition concerning the external end part of the neck. Namely, we assume that the neck meets the boundary of the external region perpendicularly  to it, and that the exterior region is concave and symmetric there (see (\ref{hypgeom}) and Figure 1). This assumption is probably purely technical and should not be necessary. However, it permits us to adapt to this
case some of the arguments of \cite{MN}, in order to obtain the lower bound after reducing the problem to an estimate near the end part of the neck. 
This reduction itself is obtained using Carleman estimates up to the boundary, as in \cite{LL,LR}.

{\bf Acknowledgements} The authors wish to thank J. Sj\"ostrand and M. Zworski for their useful suggestions, and T. Ramond for interesting discussions. We also thank T. Duyckaerts for having pointed out a mistake to us in a previous version of the manuscript.

\section{Geometrical description and results}

Consider a Helmholtz resonator in $\R^2$ consisting of a regular bounded open set  $\mathcal C$ (the cavity), connected 
to a regular unbounded open exterior domain $\mathbf E$ through a thin straight tube ${\mathcal T}(\varepsilon)$ (the neck)
of radius $\varepsilon >0$ (see figure 2). We shall suppose that $\varepsilon$ is very small.

To state this more precisely, let $\mathcal C$ and $\mathcal B$ be two bounded domains in $\R^2$ with $\mathcal C^\infty$ boundary; 
their closures and boundaries are denoted $\overline{\mathcal C}$, $\overline{\mathcal B}$ and $\partial{\mathcal C}$, 
$\partial{\mathcal B}$. We assume that Euclidean coordinates $(x,y)$ can be chosen in such a way that, for some 
$L >0$,  one has, 
\be
\label{hypgeom}
\begin{split}
&\overline{\mathcal C}\subset \mathcal B\quad; \quad (0,0)\in \partial{\mathcal C}\quad ;\quad (L,0)\in \partial{\mathcal B}\quad;\quad
{[0,L]}\times\{0\}\subset \overline{\mathcal B}\backslash {\mathcal C}\quad;\\
&
\mbox{Near } M_0:=(L,0),\, {\mathcal B} \mbox{ is convex and } \partial{\mathcal B}
\mbox{ is symmetric with}\\ & \mbox{respect to } \{y=0\}.
\end{split}
\ee
\begin{remark}\sl This also contains the case where $\partial{\mathcal B}$ is flat near $M_0$, that is when $\{ L\} \times[-\varepsilon_0,\varepsilon_0] \subset \partial{\mathcal B}$ for some $\varepsilon_0>0$.
\end{remark}
Setting ${\mathcal T}(\varepsilon):= [-\varepsilon_0, L]\times (-\varepsilon, \varepsilon) \cap (\R^2\backslash {\mathcal C})$, 
$\mathcal C(\varepsilon) = {\mathcal C} \cup \mathcal{T}(\varepsilon)$ and ${\mathbf E}:= \R^2 \backslash \overline{\mathcal B}$, then
the resonator is defined as, 
$$\Omega(\varepsilon):={\mathcal C}(\varepsilon) \cup{\mathbf E}.$$
As $\varepsilon \tendsto 0^+$, the resonator $\Omega(\varepsilon)$ collapses to $\Omega_0:={\mathcal C}\cup [0,M_0]
\cup{\mathbf E}$, where $M_0$ is the point $(L,0)\in\R^2$. 

For any domain $Q$, let $P_Q$ denote the Laplacian $-\Delta_Q$ with Dirichlet boundary conditions on $\partial Q$;
for brevity, we write $P_{\Omega_\varepsilon}$ as $P_\varepsilon$. 

The resonances of $P_\varepsilon$ are defined as the eigenvalues of the operator obtained by performing a complex dilation with
respect to the  coordinates $(x,y)$, for $|x| + |y|$ large. We are interested in those
resonances of $P_\varepsilon$ that are close to the eigenvalues of $P_{\mathcal C}$. Thus let $\lambda_0 >0$ be an eigenvalue 
of $P_{\mathcal{C}}$ with $u_0$ the  corresponding (normalized) eigenfunction. We make the following

\noindent Assumption ({\bf H}):
\begin{itemize}
\item[] $\lambda_0$ is simple;
\item[] $u_0$ does not vanish on ${\mathcal C}$ near the point $(0,0)$.
\end{itemize}
Note that these properties are automatically satisfied when $\lambda_0$ is the lowest eigenvalue of  $-\Delta_{\mathcal{C}}$. When
$\lambda_0$ is a higher eigenvalue, then the last property means that $0$ does not lie on the closure of a nodal line of $u_0$. 

By the arguments of \cite{HM}, we know that there is a resonance $\rho(\varepsilon)\in\mathbb{C}$ of  $P_\varepsilon$ such that 
$\rho(\varepsilon)\to \lambda_0$ as $\varepsilon\to 0$. Furthermore, there is an eigenvalue $\lambda(\varepsilon)$ of 
$P_{\mathcal{C}(\varepsilon)}$ such that, for any $\delta>0$, 
\be
|\rho(\varepsilon)-\lambda(\varepsilon)|\leq C_\delta e^{-\pi(1-\delta)L/\varepsilon},
\ee
for some $C_\delta>0$ and all sufficiently small $\varepsilon>0$. In particular, since $\lambda (\varepsilon) \in \mathbb R$, 
this gives 
\be
\label{upperbound}
|\im \rho(\varepsilon)|\leq C_\delta e^{-\pi(1-\delta)L/\varepsilon}.
\ee

We now state our main result.
\begin{theorem}\sl 
\label{mainth}
Under Assumption {\bf(H)}, for any $\delta >0$ there exists $C_\delta >0$ such that, for all 
$\varepsilon >0$ small enough, one has
$$
|\im \rho(\varepsilon)|\geq  \frac1{C_\delta}e^{-\pi (1+\delta)L/\varepsilon}.
$$
\end{theorem}  
\begin{remark}\sl
We extend this result to the higher dimensional case in Section \ref{dim3}. 
 \end{remark}
 \begin{remark}\sl
Gathering (\ref{upperbound}) and Theorem \ref{mainth}, we can reformulate the result as :
\be
\lim_{\varepsilon\to 0_+} \varepsilon \ln |\im \rho(\varepsilon) | = -\pi L.
\ee
 \end{remark}
 \section{Properties of the resonant state}

By definition, the resonance $\rho (\varepsilon)$ is an eigenvalue of
the complex
 distorted operator,
$$
P_\varepsilon (\mu):= U_\mu P_\varepsilon U_\mu^{-1},
$$
where $\mu>0$ is a small parameter, and $U_\mu$ is a complex
distortion of the
 form,
$$
U_\mu \varphi (x,y):= \varphi ((x,y)+i\mu f(x,y)),
$$
with $f\in C^\infty (\R^2; \R^2)$, $f= 0$ near $\overline{\mathcal B}$, $f(x,y) =(x,y)$ for $|(x,y)|
$ large enough. 
(Observe that by Weyl Perturbation Theorem, the essential spectrum of
$P_\varepsilon (\mu)$ 
is  $e^{-2i\alpha}\R_+$,
 with $\alpha =\arctan\mu$.) 

It is well known that such eigenvalues do not depend on $\mu$ (see,
e.g., \cite{SZ, HeM}), and that the corresponding eigenfunctions are
of the form $U_\mu u_\varepsilon$
 with $u_\varepsilon$ independent of $\mu$, smooth on $\R^2$ and
 analytic in a complex sector around ${\mathbf E}$. In other words, $u_\varepsilon$ is a non trivial analytic solution of the equation $-\Delta u_\varepsilon = \rho (\varepsilon) u_\varepsilon$ in $\Omega(\varepsilon)$, such that $u_\varepsilon\left\vert_{\partial\Omega(\varepsilon)}\right. =0$ and, for all $\mu >0$ small enough,  $U_\mu u_\varepsilon$ is well defined and is in $L^2(\Omega(\varepsilon))$ (in our context, this latter property will be taken as a definition of the fact that $u_\varepsilon$ is {\it outgoing}). Moreover,
 $u_\varepsilon$ can be 
normalized by setting, for some fixed $\mu>0$,
$$
\Vert U_\mu u_\varepsilon\Vert_{L^2(\Omega(\varepsilon ))} =1.
$$

In that case, we learn from \cite{HM} (in particular Proposition 3.1
and formula
 (5.13)), that, for any $\delta >0$, and for any $R>0$ large enough, one has,
\be\label{ugrand}
\|u_\varepsilon\|_{L^2(\Omega(\varepsilon)\cap\{ |(x,y)|<R\})}\geq 1- \O(e^{(\delta-\frac{\pi  L}2)/\varepsilon}),
\ee and 
\be
\label{upetit}
\|u_\varepsilon\|_{H^1( {\mathbf E}\cap\{ |(x,y)|<R\})}=  \O(e^{(\delta-\frac{\pi  L}2))/\varepsilon}).
\ee

Now, we take $R>0$ such that $\overline{\mathcal B} \subset \{ |(x,y)| <R\}$. Using the equation $-\Delta u_\varepsilon = \rho u_\varepsilon$
and Green's
 formula on the domain $\Omega(\varepsilon)\cap\{ |(x,y)|<R\}$, and using polar coordinates $(r,\theta)$, we obtain,
$$
\im\rho \int_{\Omega(\varepsilon)\cap\{ |(x,y)|<R\}  } |u_\varepsilon|^2dxdy = -\im \int_0^{2\pi} \frac {\partial u_\varepsilon}{\partial r}(R,\theta)\overline u_\varepsilon (R,\theta)Rd\theta,
$$
and thus, by (\ref{ugrand})-(\ref{upetit}), and for some $\delta_0 >0$,
\be
 \label{green}
\im\rho =-(1+\O(e^{(\delta-\pi L)/\varepsilon}))\,\,\im \int_0^{2\pi} \frac {\partial u_\varepsilon}{\partial r}(R,\theta)\overline u_\varepsilon (R,\theta)Rd\theta
\ee
where the $\O$ is locally uniform with respect to $R$.

Therefore, to prove our result, it is sufficient to obtain a lower
bound on 
$\im \int_0^{2\pi} \frac {\partial u_\varepsilon}{\partial r}(R,\theta)\overline u_\varepsilon (R,\theta)Rd\theta$. 
Note that, by using (\ref{upetit}), we immediately obtain (\ref{upperbound}).
 
\section{Estimate outside a large disc}\label{largedisc}

The goal of this section is to prove,
\begin{proposition}\sl
\label{LDisc} \sl Let $R_1>R_0>0$ be fixed in such a way that  $\overline{\mathcal B}\subset \{ |(x,y)| < R_0\}$. Then, for any $C>0$, there exists a constant $C'=C'(R_0, R_1,C)>0$ such that,  for all $\varepsilon >0$ small enough, one has,
$$
|\im\rho| \geq \frac1{C'}\Vert u_\varepsilon\Vert^2_{L^2(R_0< |(x,y)| < R_1)} - C'e^{-C/\varepsilon}.
$$
\end{proposition}

\begin{proof} Working in polar coordinates $(r,\theta)$, for $r\geq R_0$ we can represent $u=u_\varepsilon$ as,
$$
u(r,\theta)=\frac1{2\pi}\sum_{k\in\Z}u_k(r)e^{ik\theta},
$$
where $u_k(r):=\int_0^{2\pi} u(r,\theta) e^{-ik\theta}d\theta =a_kH_k(r\sqrt\rho )$, $H_k$ being the outgoing Hankel function, defined for $k\geq 0$ as
$$
H_k(t):=\frac{e^{i(t-\frac{k\pi}2-\frac{\pi}4)}}{\Gamma(k+\frac12)}\sqrt{\frac2{\pi t}}\int_0^\infty e^{-s}s^{k-\frac12}\left( 1+\frac{is}{2t}\right)^{k-\frac12}ds,
$$
 for $k<0$ by $H_k=(-1)^kH_{-k}$, and solution to,
$$
t^2H_k''(t)+tH_k'(t)+(t^2-k^2)H_k(t) =0.
$$
In particular, for all $k$, the function $h_k:=H_k(r\sqrt\rho )$ is an analytic function, solution to
\be
\label{equk}
-h_k'' -\frac1r h_k' +\frac{k^2}{r^2}h_k =\rho h_k,
\ee
and for any $\mu >0$ fixed small enough, one has,
\be
\label{uksort}
h_k(re^{i\mu})\in H^2 ([R_0, +\infty)).
\ee
By (\ref{green}), for any $R\in[R_0,R_1]$ we also have,
\be
\label{imrhoalpha}
\im\rho =-(1+\O(e^{(\delta-\pi L)/\varepsilon})) \sum_{k\in\Z}\alpha_k(R)=-(1+\O(e^{(\delta-\pi L)/\varepsilon})) \sum_{k\in\Z}\beta_k(R)|a_k|^2,
\ee
with
\be
\label{alphabeta}
\alpha_k(R):=\im R u_k'(R)\overline u_k(R)\quad ;\quad \beta_k(R):=\im R h_k'(R)\overline h_k(R).
\ee
We set,
$$
\lambda (R):= \sum_{k\in\Z}\alpha_k(R)=\sum_{k\in\Z}\beta_k(R)|a_k|^2,
$$
and, for $C>0$ arbitrary large, we write,
$$
\lambda (R)=\sum_{|k|\leq C/\varepsilon}\alpha_k(R)+\sum_{|k|> C/\varepsilon}\alpha_k(R)=:\lambda_-(R,C) +\lambda_+(R,C).
$$
We first prove,
\begin{lemma}\sl
\label{estimatlambda+}
There exists $\delta >0$ such that, for any $C>0$, one has,
$$
\lambda_+(R,C)=\O(e^{-\delta C/\varepsilon}),
$$
uniformly as $\varepsilon \to 0_+$.
\end{lemma}
\begin{proof} In view of (\ref{alphabeta}), it is enough to prove that $|u_k(R)|+|u_k'(R)| =\O(e^{-\delta |k|})$ for some $\delta =\delta (R)>0$, uniformly as $|k|\to\infty$. From (\ref{equk}), we know that $u_k$ is solution to,
$$
-k^{-2}u_k'' -\frac1{k^2r} u_k' +\frac1{r^2}u_k -\frac{\rho}{k^2} u_k=0,
$$
that can be considered as a semiclassical differential equation with small parameter $h:=|k|^{-1}$ and principal symbol $a(r,r^*):= (r^*)^2+r^{-2}$. In particular, this symbol is locally elliptic, and since $u$ is locally bounded together with all its derivatives, we also know that $u_k$ is locally uniformly bounded (together with all its derivatives) as $|k|\to\infty$. Then, we can apply standard techniques of semiclassical analysis (in particular Agmon estimates: see, e.g., \cite{Ma}) to prove that $|u_k|+|u'_k|$ is locally $\O(e^{-\delta |k|})$ for some $\delta >0$, and the result  follows.
\end{proof}

Next, we show,
\begin{lemma}\sl
\label{lemmeImrho} 
For any $C>0$ and any $\sigma\in(0,\pi L/2)$, there exists $C'=C'(C,\delta_1) >0$  such that
$$
\lambda_-(R,C) \geq \frac{1}{C'}\sum_{|k|\leq C/\varepsilon}|a_k|^2- C'|\im\rho|e^{-2\sigma/\varepsilon},
$$
uniformly as $\varepsilon \to 0_+$.
\end{lemma}
\begin{proof}
For $|k|\leq C/\varepsilon$,
let $\mu_k=\mu_{k,R}\in C^\infty (\R_+;\R_+)$ be a real non-decreasing function verifying,
$$
\mu_k(r)=0\quad {\rm for}\quad r\leq r_k:=\max(C_0|k|,R)\quad ;\quad \mu_k(r) =\frac{\mu_0}{1+|k|} \quad {\rm for}\quad r\geq r_k+1,
$$
where $\mu_0>0$ is fixed small enough, and $C_0>0$ will be chosen sufficiently large later on. We set,
\be
\label{distort}
\nu_k(r):=re^{i\mu_k(r)}\quad ; \quad g_k(r)=U_kh_k(r):=h_k(\nu_k(r)).
\ee
By (\ref{uksort}) we have,
\be
\label{vkH2}
g_k \in H^2([R_0, +\infty)).
\ee
Moreover, by construction we also have,
$$
\beta_k(R) =\im  \frac{\nu_k(R)}{\nu_k'(R)}g_k'(R)\overline g_k(R),
$$
and, by using (\ref{equk}), we see that $g_k$ is solution to,
\be
\label{eqvk}
-g_k''-\left(\frac{\nu_k'}{\nu_k}-\frac{\nu_k''}{\nu_k'}\right)g_k' + \frac{k^2(\nu_k')^2}{\nu_k^2}g_k = \rho (\nu_k')^2g_k.
\ee
Then, using (\ref{vkH2})-(\ref{eqvk}),we can write,
\begin{align*}
\beta_k(R)&= - \im\int_R^\infty \frac{d}{dr}\left(\frac{\nu_k(r)}{\nu_k'(r)}g_k'(r)\overline g_k(r)\right)dr\\
&= -\im\int_R^\infty \left[ \left(1-\frac{\nu_k\nu_k''}{(\nu_k')^2}\right)g_k'\overline g_k + \frac{\nu_k(r)}{\nu_k'(r)}g_k''(r)\overline g_k(r)+\frac{\nu_k(r)}{\nu_k'(r)}|g_k'(r)|^2\right]dr\\
&= -\im\int_R^\infty \left[ \left(\frac{k^2\nu_k'}{\nu_k}-\rho \nu_k\nu_k'\right) |g_k|^2+\frac{\nu_k(r)}{\nu_k'(r)}|g_k'(r)|^2\right]dr.
\end{align*}
Since $\nu_k'/\nu_k = r^{-1}+i\mu_k'$ and $\nu_k\nu_k' = r(1+ir\mu_k')e^{2i\mu_k}$, we obtain,
$$
\beta_k(R)=\int_R^\infty \left(\gamma_k(r)|g_k'(r)|^2+\delta_k(r)|g_k(r)|^2\right)dr,
$$
with,
\begin{align*}
\gamma_k(r):= &\frac{\mu_k'}{r^{-2}+(\mu_k')^2}\, ;\\
\delta_k(r):= &r\re\rho \sin2\mu_k+r\im\rho \cos2\mu_k\\
& +r^2\mu_k'[(\re\rho)\cos 2\mu_k-(\im\rho)\sin 2\mu_k]-k^2\mu_k'.
\end{align*}
In particular, $\gamma_k\geq 0$. Since $\mu_k\leq \mu_0(1+|k|)^{-1}$, $\im\rho\leq0$, and $\re\rho\to\lambda_0>0$ as $\varepsilon\to 0$, we also have,
$$
\delta_k \geq \delta_0r \sin2\mu_k+r\im\rho \cos2\mu_k+\mu_k'(\delta_0r^2-k^2),
$$
where $\delta_0$ is any positive constant such that $\delta_0<\lambda_0\cos 2\mu_0$. But, by construction, we have $\mu_k'(r)=0$ when $r\leq C_0|k|$. Therefore $\mu_k'(r)(\delta_0r^2-k^2)\geq \mu_k'(r)(\delta_0C_0^2-1)k^2\geq 0$ if we choose $C_0\geq \delta_0^{-1/2}$. Then, we obtain,
\be
\label{estalphak}
\begin{split}
\beta_k(R)\geq &\int_R^\infty r\left(\delta_0(\sin 2\mu_k(r)+\im\rho\cos2\mu_k(r)\right)|g_k(r)|^2dr\\
\geq &\delta_0\sin(\frac{\mu_0}{1+|k|})\int_{r_k+1}^\infty r|g_k(r)|^2dr - |\im\rho|\int_R^\infty r|g_k(r)|^2dr.
\end{split}
\ee
Since $|k|\leq C/\varepsilon$ and $|\im\rho| =\O(e^{-c_1/\varepsilon})$ for some $c_1>0$, we also have $|\im\rho|\leq \frac12\delta_0\sin(\frac{\mu_0}{1+|k|})$ for $\varepsilon>0$ small enough, and therefore,
$$
\beta_k(R)\geq \frac12\delta_0\sin(\frac{\mu_0}{1+|k|})\int_{r_k+1}^\infty r|g_k(r)|^2dr - |\im\rho|\int_R^{r_k+1} r|g_k(r)|^2dr.
$$
Equivalently, setting $v_k(r):= u_k(\nu_k(r))=a_kg_k(r)$, we have proved,
\be
\label{estalphak}
\alpha_k(R) \geq \frac12\delta_0|a_k|^2\sin(\frac{\mu_0}{1+|k|})\int_{r_k+1}^\infty r|g_k(r)|^2dr - |\im\rho|\int_R^{r_k+1} r|v_k(r)|^2dr
\ee
Now, considering a cut-off function $\chi =\chi (r)\in C^\infty (\R_+;[0,1])$ such that $\chi =1$ on $r\geq R_0$, $\chi =0$ on $r\leq R_0-\delta_0$ ($\delta_0 >0$ small enough), we see that the function $w:=\chi u$ satisfies $(-\Delta -\rho)w = [-\Delta,\chi]u$ on all of $\R^2$, and is outgoing. Then, standard estimates on the outgoing resolvent of the Laplacian (or, equivalently, on the Green function of the Helmholtz equation in $\R^n$, $n\geq 2$) show that, for all $\delta >0$ arbitrarily small, one has $w=\O(e^{\delta r}||[-\Delta,\chi]u||_{L^2})$ uniformly as $r\to\infty$. Actually, such estimates remain valid for the complex distorted Laplacian $U_0\Delta U_0^{-1}$ (where $U_0$ is as in (\ref{distort}) with some arbitrary $\mu_0\geq 0$ small enough), and since $||[-\Delta,\chi]u||_{L^2}=\O(e^{-\delta_1/\varepsilon})$ for any $\delta_1 \in (0,\pi L/2)$, we obtain: $u(r)=\O(e^{\delta r -\delta_1/\varepsilon})$ uniformly on $\{ r\in\C\, ; \re r\geq R_0\, ,\, |\im r|\leq \mu_0(\re R-R_0)\}$, where $\delta >0$ is arbitrary. In particular, this gives us: $r|v_k(r)|^2 =\O(e^{\delta r - 2\delta_1/\varepsilon})$, and therefore,
$$
\sum_{|k|\leq C/\varepsilon}\int_R^{r_k+1} r|v_k(r)|^2dr =\O\left(\frac{C}\varepsilon e^{\delta C/\varepsilon-2\delta_1/\varepsilon}\right)=\O(e^{-2\delta_1'/\varepsilon}),
$$
where $\delta'_1=\delta_1-\delta C$ can be taken arbitrarily close to $\delta_1$ (and thus, to $\pi L/2$)  by chosing $\delta <<1/C$. Inserting into (\ref{estalphak}) and taking the sum over $k$, we obtain,
\be
\label{estlambda-}
\lambda_-(R,C)\geq \frac12\delta_0\sum_{|k|\leq C/\varepsilon}|a_k|^2\sin(\frac{\mu_0}{1+|k|})\int_{r_k+1}^\infty r|g_k(r)|^2dr-C'|\im\rho |e^{-\delta_1'/\varepsilon}
\ee
with $C'=C'(C)>0$.

In order to complete the proof, we need to estimate the quantity $J_k:=\int_{r_k+1}^\infty r|g_k(r)|^2dr$ as $|k|\to \infty$. Setting $r=|k|s$, for $|k|$ large enough we find,
\be
\label{Jk}
J_k\geq|k|^2\int_{2C_0}^\infty |w_k(se^{i\mu_0/(1+|k|)})|^2ds
\ee
where $w_k(z):=z^{1/2}h_k(|k|z)$ ($z\in\C$, $|z|\geq C_0$, $|\arg z|\leq \mu_0$). Using (\ref{equk}), we see that $w_k$ is solution to,
$$
-\frac1{k^2}w_k'' +\left(\frac1{z^2}-\frac1{4k^2z^2}-\rho\right)w_k =0.
$$
This is a semiclassical Schr\"odinger equation, with small parameter $h:=|k|^{-1}$, and we can apply to it the standard WKB complex method in order to find the asymptotic of $w_k$, both as $k\to \infty$ and $\re z\to +\infty$. Using also that $w_k$ must be outgoing, we immediately obtain,
\be
\label{bkw}
w_k(z)\sim \frac{\tau_k}{(\rho-z^{-2})^{\frac14}}\exp\left(i|k|\int_{2C_0}^z (\rho-t^{-2})^{\frac12}dt\right)
\ee
as $|k|+\re z\to \infty$, uniformly with respect to $\varepsilon >0$. Here $\tau_k\in\C$ is a complex constant of normalization that we have to compute. In order to do so, we use the well-known asymptotic of $H_k(t)$ as $\re t\to +\infty$,
$$
H_k(t)\sim \sqrt{\frac2{\pi t}}\exp\left(i(t-\frac{k\pi}2-\frac{\pi}4)\right),
$$
that gives,
$$
w_k(r)=r^{\frac12}H_k(|k|r\sqrt\rho)\sim \sqrt{\frac2{\pi |k|}}\exp\left(i(|k|r\sqrt\rho-\frac{k\pi}2-\frac{\pi}4)\right)\quad (r\to +\infty).
$$
Comparing with (\ref{bkw}), we obtain,
$$
\tau_k=\rho^{\frac14} \sqrt{\frac2{\pi |k|}}e^{-i(\frac{k\pi}2+\frac{\pi}4)}e^{i|k| L}
$$
where 
$$
L:=\lim_{r\to+\infty}(r\sqrt\rho -\int_{2C_0}^r(\rho -t^{-2})^{\frac12}dt)=\lim_{r\to+\infty}(r\sqrt\rho -\left[\sqrt{\rho t^2 -1} -\tan^{-1}\sqrt{\rho t^2 -1}\right]_{2C_0}^r)
$$
that is,
$$
L=\frac{\pi}2+\sqrt{4\rho C_0^2-1}-\tan^{-1}\sqrt{4\rho C_0^2-1}.
$$
In particular, 
$$
\im L=\im \sqrt{4\rho C_0^2-1}+\frac1{2}\int_{\im \sqrt{4\rho C_0^2-1}}^{-\im \sqrt{4\rho C_0^2-1}}\frac{1}{1+(\re\sqrt{4\rho C_0^2-1}+it)^2}dt,
$$
and thus
$$
\im L =(1+\O(C_0^{-1}))\im \sqrt{4\rho C_0^2-1}\leq 0
$$
if $C_0$ has been taken sufficiently large. As a consequence,
$$
|\tau_k|\geq |\rho|^{\frac14} \sqrt{\frac2{\pi |k|}},
$$
and then, by (\ref{bkw}), and for $s\geq 2C_0$, we deduce,
$$
|k|^2|w_k(se^{i\mu_0/(1+|k|)})|^2\geq \delta_2 |k|e^{-\delta s},
$$
where $\delta_2 >0$ is a constant (independent both of $k$ and $\varepsilon$). Going back to (\ref{Jk}), for $|k|$ large enough we finally obtain,
$$
J_k \geq \frac{|k|}{C_1},
$$
where $C_1$ is a positive constant. Then, inserting into (\ref{estlambda-}), we obtain
$$
\lambda_-(R,C)\geq \frac{\delta_0}{3C_1}\sum_{|k|\leq C/\varepsilon}|a_k|^2-C'|\im\rho |e^{-\delta_1'/\varepsilon},
$$
and Lemma \ref{lemmeImrho} follows.
\end{proof}
Now, for any $K\geq 0$, we have,
$$
||u||^2_{r=R}= R\sum_{k\in\Z}|a_k|^2 |h_k(R)|^2 \leq C_K\sum_{|k|\leq K}|a_k|^2+  R\sum_{|k| >K}|a_k|^2 |h_k(R)|^2,
$$
with $C_K:=\sup_{|k|\leq K\,; \, R\in[R_0,R_1]}R|h_k(R)|^2$. Then, in the same spirit as in \cite{Bu1}, we use an estimate  on the outgoing Hankel functions that will permit us to compare its values at two different points. 
\begin{lemma}\sl 
\label{lemmeHankel}
One has,
$$
h_k(R) = -i\sqrt{\frac{2}{\pi}}\, k^{k-\frac12}\left(\frac2{eR\sqrt{\rho}}\right)^k \left( 1+\O(k^{-1})\right),
$$
uniformly with respect to $R\in[R_0,R_1]$, $\varepsilon>0$ small enough, and $k\geq 1$ large enough.
\end{lemma}
\begin{proof} See Appendix.
\end{proof}

 It follows from this lemma that, for any $R\in [R_0,R_1]$, we have,
$$
\frac{|h_k(R)|} {|h_k(R_0)|}={\mathcal O}\left( (R_0/R)^{|k|}\right)
$$
uniformly as $|k|\to \infty$. Therefore, we obtain,
\be
||u||^2_{r=R} \leq C_K\sum_{|k|\leq K}|a_k|^2+CR\sum_{|k| >K}|a_k|^2 |h_k(R_0)|^2R_0^{2|k|}R^{-2|k|}
\ee
where $C>0$ does not depend on $K,R$. Integrating with respect to $R$ on the interval $[R_0,R_1]$, we obtain,
$$
||u||^2_{R_0\leq R\leq R_1}\leq C_K'\sum_{|k|\leq K}|a_k|^2+C\sum_{|k| >K}|a_k|^2 |h_k(R_0)|^2R_0^{2|k|}\frac{R_0^{2-2|k|}}{2|k|-2},
$$
and thus,
\be
\label{uRuR'}
||u||^2_{R_0\leq R\leq R_1}\leq C_K'\sum_{|k|\leq K}|a_k|^2+\frac{CR_0}{2K-2} ||u||^2_{r=R_0}.
\ee
Moreover, for all $S\in[R_0,R_1]$, we have,
$$
||u||^2_{r=R_0}=||u||^2_{r=S}-\int_{S}^{R_0}(||u(r)||^2_{L^2(0,2\pi )}+2r\re\la \partial_r u,  u\ra_{L^2(0,2\pi )}) dr,
$$
that gives,
$$
||u||^2_{r=R_0}= ||u||^2_{r=S}+{\mathcal O}(|| \partial_r u||^2_{R_0\leq r\leq R_1}+||u||^2_{R_0\leq r\leq R_1}),
$$
and thus, using the equation $-\Delta u=\rho u$ and standard Sobolev estimates,
$$
||u||^2_{r=R_0}= ||u||^2_{r=S}+{\mathcal O}(||u||^2_{R_0\leq r\leq R_1}).
$$
Inserting this into (\ref{uRuR'}), and taking $K$ sufficiently large, we obtain,
\be
\label{uRuR0R1}
||u||^2_{R_0\leq R\leq R_1} \leq C'_K\sum_{|k|\leq K}|a_k|^2+\frac{C'}{K-1}||u||^2_{r=S},
\ee
where $C', C'_K>0$ are constants, and $C'$ is independent of $K$. Finally, integrating in $S$ on $[R_0,R_1]$, and increasing again the value of $K$, we arrive to,
\be
\label{uR0R1ak}
||u||^2_{R_0\leq r\leq R_1} \leq 2C'_K\sum_{|k|\leq K}|a_k|^2.
\ee
Then, Proposition \ref{LDisc} directly follows from (\ref{imrhoalpha}), Lemma \ref{estimatlambda+}, Lemma \ref{lemmeImrho}, and (\ref{uR0R1ak}).
 \end{proof}
 
 \begin{remark}\sl By integrating with respect to $R$ on any bounded interval of $[R_0, +\infty)$, and by using the equation $-\Delta u_\varepsilon = \rho u_\varepsilon$ and standard estimates on the Laplacian, we easily deduce from this proposition that, for any bounded open set $V\subset \{ |(x,y)|\geq R_0\}$ and any $s\geq 0$, one has $\Vert u_\varepsilon\Vert^2_{H^s(V)}=\O(|\im\rho|+e^{-C/\varepsilon})$ for any $C>0$.
 \end{remark}
 
  \begin{remark}\sl The result of Proposition \ref{LDisc} can easily be generalized to any dimension $n\geq 2$ by working with the complex measure $(\nu_k(r)/\nu'_k(r))^{n-1}dr$ instead of  $(\nu_k(r)/\nu'_k(r))dr$ in the proof of Lemma  \ref{lemmeImrho}.\end{remark}
  
\begin{remark}\sl As pointed out to us by J. Sj\"ostrand, an alternative (and probably more conceptual) proof of Proposition \ref{LDisc} may consists in making the change of scale $r\mapsto r/h$, where $h>0$ is an extra small parameter, and to apply the techniques of semiclassical analysis as $h\to 0_+$. The fact that $u$ is outgoing means that it lives around the outgoing trajectories starting from the obstacle, and thus in a microlocal weighted space where $-h^2\Delta -\rho$ can be written as the product of an elliptic pseudodifferential operator with $\partial_r -iA$, where the selfadjoint operator $A$ acts on the tangent variable $\theta$ only, and is positive. Such arguments are developed in \cite{Sj}, Section 4.
\end{remark}

\section{Estimate near the obstacle}\label{eastob}
Now, reasoning by contradiction, assume the existence of $\delta_0 >0$ such that, along a sequence $\varepsilon \to 0^+$, one has
\be
\label{absurd}
|\im\rho| =\O(e^{-(\pi  L+\delta_0)/\varepsilon}).
\ee 
In the rest of the proof, it will always been assumed that $\varepsilon$ tends to zero along this sequence. Then Proposition \ref{LDisc} (added to standard Sobolev estimates) tells us that for any $R_1>R_0>0$ such that $\overline{\mathcal B}\subset \{ |(x,y)| < R_0\}$, we have,
\be
\label{decu}
\Vert u_\varepsilon\Vert^2_{H^1(R_0<|(x,y)| <R_1)} = \O(e^{-(\pi  L+\delta_0)/\varepsilon}).
\ee
To propagate this estimate up to an arbitrarily small neighborhood of $\overline{\mathcal B}$, we use the Carleman estimate 
in \cite[Theorem 3.5]{LL}.

First fix a point $(x_0,y_0)$ in ${\mathbf E}=\R^2\backslash \overline{\mathcal B}$, and assume there exists a real function $f$ defined on a small  open neighborhood $V_0$ of $(x_0,y_0)$ in ${\mathbf E}$, with $f(x_0,y_0)=0$, $\nabla f(x_0,y_0)\not= 0$, and such that for any $\delta >0$ small enough, there exists $\delta' =\delta'(\delta)>0$, such that,
\be
\label{estaprio}
\Vert u_\varepsilon\Vert^2_{H^1(V\cap \{f\geq\delta\})}=\O (e^{-(\pi  L+\delta')/\varepsilon}),
\ee
uniformly as $\varepsilon \to 0_+$. (For instance, in view of (\ref{decu}), $(x_0,y_0)$ could be any point of ${\mathbf E}$ such that $|(x_0,y_0)|=R_-$, with $R_-:=\inf\{ R>0\,; \, \overline{\mathcal B}\subset \{ |(x,y)| \leq R\}$, and $f(x,y) = x^2 +y^2 - R_-^2$.)

For $\lambda >0$ fixed large enough and $(x,y)$ in $V_0$, following \cite{LL, LR} we consider the function,
$$
\varphi(x,y):= e^{\lambda (f(x,y) - (x-x_0)^2 -(y-y_0)^2)}.
$$
Then, setting,
$$
p_\varphi (x,y,\xi,\eta):= \xi^2+\eta^2 - |\nabla\varphi (x,y)|^2+2i\la \nabla\varphi (x,y), (\xi,\eta)\ra = q_1+iq_2,
$$
it is easy to check that, if $\lambda$ has been taken large enough, then there exists a constant $C_0>0$ such that one has the implication,
$$
p_\varphi (x,y,\xi,\eta) =0 \, \Rightarrow \, \{ q_1, q_2\} (x,y,\xi,\eta) \geq \frac1{C_0},
$$
where $\{ q_1, q_2\}$ is the Poisson bracket of the real-valued functions $q_1$ and $q_2$. Moreover, possibly by shrinking $V_0$ around $(x_0,y_0)$, we see that $\nabla\varphi \not =0$ on $V$. In particular, Assumption 3.1 of \cite{LL} is satisfied, and if
 $\chi\in C_0^\infty (V_0\, ;\, [0,1])$ is such that $\chi =1$ near $(x_0,y_0)$, we can apply Theorem 3.5 of \cite{LL} to the function $w:=\chi u_\varepsilon$, and with small parameter $h:=\varepsilon / \mu$, where $\mu >0$ is an extra-parameter that will be fixed small enough later on. Then, for $\varepsilon/\mu$ small enough, we obtain,
\be
\label{carlem1}
\Vert e^{\mu\varphi /\varepsilon}w\Vert_{L^2}^2 + \mu^{-2}\varepsilon^2\Vert  e^{\mu\varphi /\varepsilon}\nabla w\Vert_{L^2}^2\leq C\mu^{-3}\varepsilon^3\Vert  e^{\mu\varphi /\varepsilon}\Delta w\Vert_{L^2}^2
 \ee
where $C>0$ is a constant. Then, writing $-\Delta w = \rho w -[\Delta, \chi] u_\varepsilon$, and observing that, for $\varepsilon/\mu$ small enough, the term involving $\rho w$ in the right-hand side of (\ref{carlem1}) can be absorbed by the first term of the left-hand side, we are led to,
$$
\Vert e^{\mu\varphi /\varepsilon}w\Vert_{L^2}^2 + \mu^{-2}\varepsilon^2\Vert  e^{\mu\varphi /\varepsilon}\nabla w\Vert_{L^2}^2\leq C\mu^{-3}\varepsilon^3\Vert  e^{\mu\varphi /\varepsilon}[\Delta, \chi] u_\varepsilon\Vert_{L^2}^2,
$$
with a new constant $C>0$. Now, setting $m_0:=\sup_{V_0} \varphi$, $V_0':=\{ \chi =1\}$, $S_{\delta}:= {\rm Supp}\nabla\chi \cap \{ f<\delta\}$ ($\delta >0$ small enough), and using (\ref{estaprio}), we deduce,
\be
\begin{split}
\Vert e^{\mu\varphi /\varepsilon}u_\varepsilon\Vert_{L^2(V_0')}^2 + &\mu^{-2}\varepsilon^2\Vert  e^{\mu\varphi /\varepsilon}\nabla u_\varepsilon\Vert_{L^2(V_0')}^2\\
&=\O(\mu^{-3}\varepsilon^3\Vert  e^{\mu\varphi /\varepsilon}[\Delta, \chi] u_\varepsilon\Vert_{L^2(S_\delta)}^2+ e^{(\mu m_0-\pi  L-\delta')/\varepsilon}).
\end{split}
\ee

On the other hand, we have $S_\delta\subset \{f<\delta\}\cap \{|(x,y)-(x_0,y_0)| \geq \delta_1\}$ for some $\delta_1 >0$ independent of $\delta$, and thus, by construction, for $\delta >0$ sufficiently small, there exists a constant $\delta_2>0$ such that,
\be
\label{estSdelta}
S_{\delta}\subset \{ \varphi (x,y) \leq 1-\delta_2\}. 
\ee
As a consequence, we obtain,
\be
\label{estcarlmu}
\begin{split}
\Vert e^{\mu\varphi /\varepsilon}u_\varepsilon\Vert_{L^2(V_0')}^2 + &\mu^{-2}\varepsilon^2\Vert  e^{\mu\varphi /\varepsilon}\nabla u_\varepsilon\Vert_{L^2(V_0')}^2\\
&=\O(\mu^{-3}\varepsilon^3e^{\mu(1-\delta_2)/\varepsilon}\Vert   u_\varepsilon\Vert_{H^1(S_\delta)}^2+ e^{(\mu m_0-\pi  L-\delta')/\varepsilon}).
\end{split}
\ee
Since $S_\delta\subset {\mathbf E}$, we also know (see (\ref{upetit})) that $\Vert   u_\varepsilon\Vert_{H^1(S)}$ is not exponentially larger than $e^{-\pi L/2\varepsilon}$. Moreover, since $\varphi (x_0,y_0) =1$, if $B_r$ stands for the ball of radius $r$ centered at $(x_0,y_0)$, we have $\varphi \leq 1-\theta (r)$ on $B_r$, with $\theta (r)\to 0$ as $r\to 0$. Therefore, for $r>0$ small enough, we deduce from (\ref{estcarlmu}),
\be
\label{estcarlmu2}
\begin{split}
\Vert u_\varepsilon\Vert_{L^2(B_r)}^2 + &\mu^{-2}\varepsilon^2\Vert  \nabla u_\varepsilon\Vert_{L^2(B_r)}^2\\
&=\O(\mu^{-3}\varepsilon^3e^{(\mu(\theta (r)-\frac12\delta_2)-\pi L)/\varepsilon}+ e^{(\mu (m_0-1+\theta(r))-\pi  L-\delta')/\varepsilon}).
\end{split}
\ee
Now, we first fix $\delta>0$ such that (\ref{estSdelta}) is satisfied, and then $r>0$ and $\mu>0$ sufficiently small, in such a way that $\theta (r)\leq \frac14\delta_2$ and $(\mu (m_0-1+\theta(r))\leq \frac12\delta'$. We obtain,
$$
\Vert u_\varepsilon\Vert_{L^2(B_r)}^2 + \varepsilon^2\Vert  \nabla u_\varepsilon\Vert_{L^2(B_r)}^2=\O(e^{-\pi L/\varepsilon} (e^{-\frac{\mu}4\delta_2/\varepsilon}+ e^{-\frac12\delta'/\varepsilon}))
$$
In other words, we have extended the estimate (\ref{estaprio}) across the boundary $\{f=0\}$ near $(x_0,y_0)$. Our argument can be performed near any point $(x_0,y_0)\in {\mathbf E}$ where an estimate like (\ref{estaprio}) is valid, and thus, starting form the points of the circle $\{ |(x,y)| =R_-\}$ (where the estimate is valid thanks to Proposition \ref{LDisc} and to the assumption (\ref{absurd})), and deforming continuously this circle up to make it become the boundary of $\mathcal B$, a standard covering argument leads to,
\begin{proposition}\sl
\label{nearobst}
\label{boundary}
Under assumption (\ref{absurd}), for any compact set $K\subset {\mathbf E}$, there exists $\delta =\delta(K)>0$ such that,
$$
\Vert u_\varepsilon\Vert_{H^1(K)}^2 =\O(e^{-(\pi  L+\delta)/\varepsilon}),
$$
uniformly as $\varepsilon \to 0_+$.
\end{proposition}  
\begin{remark}\sl
\label{remHs}
By using the equation, we deduce that, actually, in the previous estimate $H^1$ can be replaced by any $H^m$, $m\geq 0$.
\end{remark}

\section{Estimate at the boundary}\label{eastbound}

Now, we plan to propagate the estimates of the previous section up to the boundary of $\mathcal B$ (but away from any arbitrarily small neighborhood of $M_0$), by making use of the Carleman estimate at the boundary as stated in \cite{LR}, Proposition 2 (see also \cite{LL}, Theorem 7.6, applied to $e^{-\rho t} u_\varepsilon(x,y)$). 

We consider an arbitrary point $(x_0,y_0)$ on the boundary $\partial{\mathcal B}$ of $\mathcal B$, with $(x_0,y_0)\not= (L,0)$, and a small enough open neighborhood $V$ of $(x_0,y_0)$ in $\R^2$. We also consider a compact neighborhood $K\subset V$ of $(x_0,y_0)$, and we denote by $f$ a function defining $\partial{\mathcal B}$ near $(x_0,y_0)$, in the sense that one has,
$$
{\mathcal B}\cap V =\{ (x,y)\in V\, ; \, f(x,y)<0\},
$$
and $\nabla f \not= 0$ on $V$. Finally, as in following \cite{LL, LR}, one sets,
$$
\varphi (x,y):= e^{\lambda (f(x,y) - (x-x_0)^2-(y-y_0)^2)},
$$
where $\lambda >0$ is fixed sufficiently large and $C_0>\sup_V (f(x,y) - (x-x_0)^2-(y-y_0)^2))$. In particular, if $V$ has been taken sufficiently small, we see (e.g. as in \cite{LL}, Lemma A.1) that $\varphi$ satisfies Assumption (8) of \cite{LR}. Moreover, since the outward pointing unit normal to $\mathbf E$ in $V$ is $n:=-\nabla f/ |\nabla f|$, we also have $\partial_n\varphi \left|_{\partial{\mathbf E}\cap V}\right. <0$. Therefore, we can apply Proposition 2 of \cite{LR} (or, alternatively, Theorem 7.6 of \cite{LL}), and we obtain the existence of a constant $C>0$ such that, for any $\mu, \varepsilon>0$ with $\varepsilon /\mu$ small enough,
\begin{eqnarray*}
&&\Vert e^{\mu\varphi /\varepsilon}\chi u_\varepsilon\Vert_{L^2({\mathbf E}\cap V)}^2+\mu^{-2}\varepsilon^2\Vert e^{\mu\varphi /\varepsilon}\nabla (\chi u_\varepsilon )\Vert_{L^2({\mathbf E}\cap V)}^2\\
&& \hskip 5cm\leq C\mu^{-3}\varepsilon^3\Vert e^{\mu\varphi /\varepsilon}\Delta(\chi u_\varepsilon)\Vert_{L^2({\mathbf E}\cap V)}^2,
\end{eqnarray*}
where $\chi \in C_0^\infty (V\, ;\, [0,1])$ is some fixed cut-off function such that $\chi =1$ on $K$. Using that $-\Delta u_\varepsilon =\rho u_\varepsilon$, for $\varepsilon$ small enough, we deduce,
$$
\Vert e^{\mu\varphi /\varepsilon}u_\varepsilon\Vert_{L^2( {\mathbf E}\cap K)}^2+\mu^{-2}\varepsilon^2\Vert e^{\mu\varphi /\varepsilon}\nabla u_\varepsilon\Vert_{L^2({\mathbf E}\cap K)}^2\leq 2C\mu^{-3}\varepsilon^3\Vert  e^{\mu\varphi /\varepsilon}[\Delta ,\chi] u_\varepsilon)\Vert_{L^2({\mathbf E}\cap V)}^2.
$$
Now, for all $\delta >0$ small enough, on ${\rm Supp} \nabla\chi \cap \{ f\leq \delta\}\cap V$, we have,
$$
\varphi \leq \varphi (x_0,y_0) - \delta',
$$
with $\delta' =\delta'(\delta) >0$. On the other hand, on $\{ f\geq \delta\}\cap V$, by Proposition \ref{nearobst} we have,
$$
\Vert u_\varepsilon\Vert_{L^2(\{ f\geq \delta\}\cap V)}^2 =\O(e^{-(\pi L+\delta')/\varepsilon}).
$$
Therefore, using also (\ref{upetit}), and fixing $\mu>0$ in a convenient way as before, we obtain the existence of $\delta_1>0$, such that,
$$
\Vert e^{\mu\varphi /\varepsilon}u_\varepsilon\Vert_{L^2( {\mathbf E}\cap K)}^2+\varepsilon^2\Vert e^{\mu\varphi /\varepsilon}\nabla u_\varepsilon\Vert_{L^2({\mathbf E}\cap K)}^2 =\O(e^{(\mu\varphi (x_0,y_0)-\pi L -\delta_1)/\varepsilon}),
$$ 
and, if $V'\subset K$ is a sufficiently small neighborhood of $(x_0,y_0)$, we finally obtain,
$$
\Vert u_\varepsilon\Vert_{H^1( {\mathbf E}\cap V')}^2=\O(e^{-(\pi L +\frac12\delta_1)/\varepsilon}).
$$
Since $(x_0,y_0)$ was arbitrary on $\partial{\mathcal B}\backslash \{ M_0\}$ (where $M_0=(L,0)$), we have proved,
\begin{proposition}\sl
\label{estbound}
Under the assumption (\ref{absurd}), for any neighborhood $\mathcal U$ of $M_0$ and any compact set $K\subset \R^2$, there exists $\delta >0$ such that,
$$
\Vert u_\varepsilon\Vert_{H^1({\mathbf E}\cap K \backslash {\mathcal U})}^2 =\O(e^{-(\pi  L+\delta)/\varepsilon}),
$$
uniformly as $\varepsilon \to 0_+$.
\end{proposition}
\begin{remark}\sl
\label{remHsbound}
By using the equation and a standard result of regularity on the Dirichlet Laplacian (see, e.g., \cite{Br}), we can deduce that, in the previous estimate, $H^1$ can be replaced by any $H^m$, $m\geq 0$.
\end{remark}

\section{Estimate near the aperture}

Now, we concentrate our attention to a small neighborhood of $M_0$ in $\overline{\mathbf E}$. More precisely, we fix $\varepsilon_1\in (0,\varepsilon_0]$, such that,
$$
\frac{\pi^2}{4\varepsilon_1^2} > \lambda_0,
$$
and we consider the rectangle,
$$
Q:= [L_\varepsilon, L+\varepsilon_1]\times [-\varepsilon_1, \varepsilon_1],
$$
where $L_\varepsilon =L-{\mathcal O}(\varepsilon^2)$ is defined as the unique value such that $(L_\varepsilon, \pm \varepsilon)\in \partial{\mathcal B}$. 

In particular, the point $M_\varepsilon:=(L_\varepsilon, 0)$ belongs to $\partial Q$, and, if $\varepsilon_1$ is taken sufficiently small, then,
$$
Q\backslash (\{L_\varepsilon\}\times [-\varepsilon_1, \varepsilon_1])\subset \Omega(\varepsilon).
$$

Moreover, by Proposition \ref{estbound}, we know the existence of some $\delta >0$ such that $u_\varepsilon$ is $\O(e^{-(\pi  L+\delta)/\varepsilon})$ near $\partial Q\backslash (\{L_\varepsilon\}\times [-\varepsilon_1, \varepsilon_1])$.

Let $\chi\in C_0^\infty (\R^2;[0,1])$ such that (see Figure 2),
\begin{itemize}
\item $\chi = 1$ on $[L_\varepsilon, L+\frac12\varepsilon_1]\times [-\frac12\varepsilon_1, \frac12\varepsilon_1]$;
\item $\chi =0$ on $([L+\varepsilon_1, +\infty)\times \R) \cup (\R\times (-\infty, -\varepsilon_1])\cup (\R \times [\varepsilon_1, +\infty)) $.
\end{itemize}
\begin{figure}[h]
\label{fig2}

\vspace*{0ex}
\centering
\includegraphics[width=10cm,angle=0]{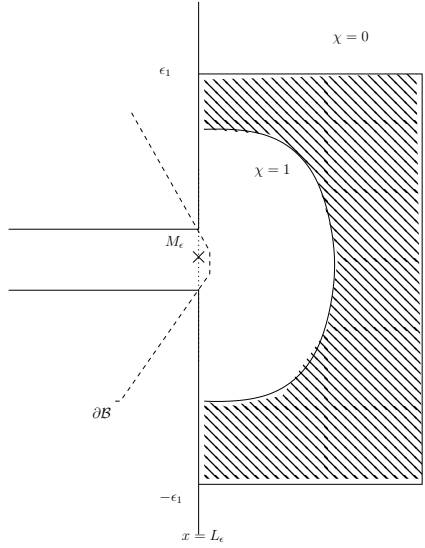}
\vspace*{0ex}
\caption{The aperture.}
\end{figure}

We set,
$$
v:= \chi u_\varepsilon.
$$
In particular, $v\in H^2 (Q)$, and $v\left|_{|y| =\varepsilon_1}\right. =0$. Therefore, on $Q$, we can expand $v$ as,
\be
\label{repv}
v(x,y) =\sum_{j\geq 1} v_j(x)\varphi_j (y),
\ee
where the $\varphi_j$'s are the eigenfunctions of the Dirichlet realization of $-d^2/dy^2$ on $[-\varepsilon_1, \varepsilon_1]$, namely,
$$
\varphi_{2j}(y) =\frac1{\sqrt{\varepsilon_1}} \sin(\alpha_{2j}y/\varepsilon_1)\, ;\, \varphi_{2j-1}(y) = \frac1{\sqrt{\varepsilon_1}}\cos(\alpha_{2j-1}y/\varepsilon_1)\, ;\, \alpha_j:= \frac{j\pi}2,
$$
and $v_j\in H^2([L_\varepsilon,L+\varepsilon_1])$. 
Moreover, using Proposition \ref{estbound} and Remark \ref{remHsbound}, on $Q$ we have,
$$
-\Delta v =\rho v + r
$$
where $\Vert r\Vert^2_{H^m(Q)}=\Vert [\Delta, \chi]u_\varepsilon\Vert^2_{H^m(Q)}=\O(e^{-(\pi  L+\delta)/\varepsilon})$, and $r\left|_{|y| =\varepsilon_1}\right. =0$ ($m\geq 0$ arbitrary, and $\delta =\delta (m) >0$). 
We deduce that the $v_j$'s verify,
\be
\label{eqvj}
-v_j'' +\beta_j v_j = r_j,
\ee
where we have set $\beta_j:= \frac{\alpha_j^2}{\varepsilon_1^2} -\rho$, and $r_j:=\int_{-\varepsilon_1}^{\varepsilon_1} r(x,y)\varphi_j(y)dy$, so that we have,
\be
\label{estrj}
\sum_{j\geq 1}j^m\Vert r_j\Vert^2_{H^m([L,L+\varepsilon_1])}=\O(e^{-(\pi  L+\delta)/\varepsilon}).
\ee
 By construction, we also have $v_j =0$ on $[L+\varepsilon_1, +\infty)$.

\begin{proposition}\sl 
\label{propb} Assume (\ref{absurd}). Then,
for all $j\geq 1$, there exist $b_j\in \C$ and $s_j\in \cap_{m\geq 0}H^m([L,L+\varepsilon_1]$, such that,
\begin{eqnarray*} 
&& v_j(x) = b_je^{-(x-L_\varepsilon)\sqrt{\beta_j}} + s_j(x);\\
&& \sum_{j\geq 1}j^m\Vert s_j\Vert^2_{H^m([L_\varepsilon,L+\varepsilon_1])}=\O(e^{-(\pi  L+\delta_m)/\varepsilon}),
\end{eqnarray*}
with $\delta_m >0$ and uniformly with respect to $\varepsilon$ small enough.
\end{proposition}
\begin{proof} Set,
$$
 W_j:= \left(\begin{array}{c} v_j\\ v'_j\end{array}\right).
$$
Then, by (\ref{eqvj}), $W_j$ is solution of,
$$
\left\{
\begin{array}{l}
W_j' = A_j W_j -R_j;\\
W_j(L+\varepsilon_1) =0,
\end{array}\right.
$$
with $A_j:=\left(\begin{array}{cc} 0 & 1\\ \beta_j & 0\end{array}\right)$ and $R_j:=\left(\begin{array}{c} 0\\ r_j\end{array}\right)$. Therefore,
$$
W_j (x) =\int_x^{L+\varepsilon_1} e^{(x-t)A_j}R_j(t)dt,
$$
and, diagonalizing $A_j$ and re-writing the solution in a basis of eigenvectors of $A_j$, we obtain in particular,
$$
v'_j(x)+\sqrt{\beta_j}v_j(x)=\int_x^{L+\varepsilon_1}e^{(x-t)\sqrt{\beta_j}}r_j(t)dt.
$$
Using again that $v(L+\varepsilon_1)=0$, we deduce,
$$
v_j(x) =-\int_x^{L+\varepsilon_1}\int_{x_1}^{L+\varepsilon_1}e^{(2x_1-t-x)\sqrt{\beta_j}}r_j(t)dtdx_1.
$$ 
Then, the results follows with $b_j:=-\int_{L_\varepsilon}^{L+\varepsilon_1}\int_{x_1}^{L+\varepsilon_1}e^{(2x_1-t-L_\varepsilon)\sqrt{\beta_j}}r_j(t)dtdx_1$ and $s_j(x):=\int_{L_\varepsilon}^{x}\int_{x_1}^{L+\varepsilon_1}e^{(2x_1-t-x)\sqrt{\beta_j}}r_j(t)dtdx_1$, by observing that $\re((2x_1-t-x)\sqrt{\beta_j})<0$ on the domain of integration of $s_j(x)$ and by using (\ref{estrj}).
\end{proof}

\begin{remark}\sl
Let $\varepsilon_2 \in (0, \frac12\varepsilon_1)$ arbitrary. By Proposition \ref{nearobst}, we know that there exists a constant $\delta =\delta(\varepsilon_2) >0$ such that,
$$
\|v\|_{L^2((L+\varepsilon_2,L+\varepsilon_1)\times (-\varepsilon_1,\varepsilon_1))}=\O(e^{-(\pi  L+\delta)/2\varepsilon}).
$$
On the other hand, using
(\ref{repv}) and Proposition \ref{propb}, on $(L_\varepsilon,L+\varepsilon_1)\times (-\varepsilon_1,\varepsilon_1))$, we have,
$$
v(x,y) =\sum_{j\geq 1} b_je^{-(x-L_\varepsilon)\sqrt{\beta_j}}\varphi_j(y) + s(x,y),
$$
with $\| s\|_{L^2((L_\varepsilon,L+\varepsilon_1)\times (-\varepsilon_1,\varepsilon_1))}=\O(e^{-(\pi  L+\delta_0)/2\varepsilon})$ for some constant $\delta_0 >0$. Since $\sqrt \beta_j\sim \frac{j\pi}{2\varepsilon_1}$ as $j\to\infty$, and $\varepsilon_2$ is arbitrarilly small, we immediately deduce that, for any $\nu>0$, there exists $\delta =\delta(\nu)>0$, such that,
\be
\label{estbj1}
\sum_{j\geq 1} |b_j|^2 e^{ -\nu j }=\O(e^{-(\pi  L+\delta)/\varepsilon}),
\ee
uniformly as $\varepsilon \to 0_+$.
  \end{remark}
  
 \section{Representations at the aperture}
In this section, we consider the trace of $v$ on $\{ x=L_\varepsilon\}$. By construction, it also coincides with the trace $u_\varepsilon$ as long as $|y| < \frac12\varepsilon_1$. Now, as in \cite{MN}, there are two ways of taking this trace, depending if one takes the limit $x\to (L_\varepsilon)_+$ or $x\to (L_\varepsilon)_-$.

Considering first the limit $x\to (L_\varepsilon)_-$, we can just apply the results of \cite{MN}, Sections 4 \& 6 (in particular (4.2), (4.3) and Lemma 6.1), and, for $x<L_\varepsilon$ close to $L_\varepsilon$ and $|y| < \varepsilon$, we obtain,
\be
\label{letube}
v(x,y) = \sum_{k=1}^\infty \left( a_{k,+} e^{\theta_k x/\varepsilon}+a_{k,-} e^{-\theta_k x/\varepsilon}\right)\psi_k(y),
\ee
where we have used the notations,
\begin{eqnarray*}
&& \psi_{2k}(y) =\frac1{\sqrt{\varepsilon}} \sin(\alpha_{2k}y/\varepsilon)\, ;\, \psi_{2k-1}(y) = \frac1{\sqrt{\varepsilon}}\cos(\alpha_{2k-1}y/ \varepsilon)\, ;\, \alpha_k:= \frac{k\pi}2;\\
&&\theta_k:=\sqrt{\alpha_k^2-\varepsilon^2\rho(\varepsilon)},
\end{eqnarray*}
(here $\sqrt{\cdot}$ stands for the principal square root), and where $a_{k,\pm}$ are ($\varepsilon$-dependent) constant complex numbers. Moreover, the sum converges in $H^2((L-\varepsilon_1, L_\varepsilon)\times (-\varepsilon, \varepsilon))$, and the limit $x\to (L_\varepsilon)_-$ gives (see \cite{MN}, Lemma 6.1),
\be
\label{match1}
v(L_\varepsilon,y) = \sum_{k=1}^\infty \left( a_{k,+} e^{\theta_k L_\varepsilon/\varepsilon}+a_{k,-} e^{-\theta_k L_\varepsilon/\varepsilon}\right)\psi_k(y),
\ee
together with (see \cite{MN}, formula (6.7)),
\be
\partial_x v (L,y) = \frac1{\varepsilon}\sum_{k\geq 1}\theta_k\left( a_{k,+} e^{\theta_k L_\varepsilon/\varepsilon}-a_{k,-} e^{-\theta_k L_\varepsilon/\varepsilon}\right)\psi_k(y)\, \mbox{ in } H^{1/2}(|y|\leq\varepsilon).
\ee

Then, starting from(\ref{repv}), and using similar arguments,  the limit $x\to (L_\varepsilon)_+$ can be taken in the same way, and, using Proposition \ref{propb}, we obtain,
\be
v(L_\varepsilon,y) = \sum_{j= 1}^\infty \left(b_j+s_j(L_\varepsilon)\right)\varphi_j (y),
\ee
together with,
\be
 \partial_x v (L_\varepsilon,y) =  \sum_{j= 1}^\infty (-\sqrt{\beta_j} b_j+{s_j}'(L_\varepsilon)) \varphi_j (y)
 \mbox{ in } H^{1/2}(|y|\leq \varepsilon_1).
 \ee
Moreover, still by Proposition \ref{propb}, we have,
\be
\label{match2}
\sum_{j\geq 1}\left( |s_j(L_\varepsilon)|^2 + |s_j'(L_\varepsilon)|^2\right) =\O(e^{-(\pi L +\delta )/\varepsilon}),
\ee
for some constant $\delta >0$.

\section{Estimates on the coefficients }

At this point, we can proceed as \cite{MN}, Section 7 (but working with $v$ instead of $u_\varepsilon$), with the difference that, in our present case, the index $j_0$ appearing in \cite{MN}, formula (6.8), is just 0 (that is, all the sums over $\{j\leq j_0\}$ become null). For the sake of completeness, we briefly reproduce these arguments here.

The main idea consists in computing in two different ways the three following quantities:
$$
\la v, \partial_x v\ra_{\{L_\varepsilon\}\times[-\varepsilon,\varepsilon]}\, , \, \la v,\varphi_1\ra_{\{L_\varepsilon\}\times[-\varepsilon,\varepsilon]}\, , \, \la\partial_x v,\psi_1\ra_{\{L\}\times[-\varepsilon,\varepsilon]}.
$$
We set
$$
A_{k,\pm}:= a_{k,\pm}e^{\pm \theta_k L/\varepsilon}.
$$

In view of (\ref{match1})-(\ref{match2}),  the two computations of  $\la v, \partial_x v\ra_{\{L_\varepsilon\}\times[-\varepsilon,\varepsilon]}$ give the identity
$$\frac1{\varepsilon}\sum_{k\geq 1}\theta_k(|A_{k,+}|^2-|A_{k,-}|^2+2i\im (A_{k,+}\overline{A}_{k,-}))
=-\sum_{j\geq 1}(\sqrt{\beta_j })|b_j|^2+r(\varepsilon),
$$
with
\be
\begin{split}
r(\varepsilon) &=\O(e^{-(\pi L +\delta )/\varepsilon} + e^{-(\pi L +\delta )/2\varepsilon}(\sum_{j\geq 1}|b_j|^2)^{\frac12})\\
&=\O(e^{-(\pi L +\frac{\delta}2 )/\varepsilon} + e^{-\delta /\varepsilon}\sum_{j\geq 1}|b_j|^2).
\end{split}
\ee
Taking the real part, and using the fact that $\re\theta_k \sim k\pi /2$ as $k\to\infty$, while $|\im\theta_k|={\mathcal O}(k^{-1}e^{-\delta /\varepsilon})$ for some constant $\delta >0$,  we obtain, 

\begin{eqnarray*}
&& \frac1{\varepsilon}\sum_{k\geq 1}(\re\theta_k)(|A_{k,+}|^2-|A_{k,-}|^2) +\frac1{\varepsilon}\sum_{k\geq 1}{\mathcal O}(k^{-1}e^{-\delta/\varepsilon}) |A_{k,+}A_{k,-}|\\
&&\hskip 6cm = -\sum_{j\geq 1}(\re\sqrt{\beta_j })|b_{j}|^2 + r(\varepsilon).
\end{eqnarray*}
In particular, since $\re\sqrt{\beta_j }= \frac{\pi  j}{2\varepsilon_1}(1+{\mathcal O}(\varepsilon^2 j^{-2}))$, we see that there exists a constant $C>0$ such that
\begin{equation}
\begin{split}
\label{est1}
  \sum_{k\geq 1}\re\theta_k(|A_{k,+}|^2-|A_{k,-}|^2) &
\leq  C\sum_{k\geq 1}k^{-1}e^{-\delta/\varepsilon} |A_{k,+}A_{k,-}| \\
&-\frac{\pi}2\frac{\varepsilon}{\varepsilon_1} \sum_{j\geq 1 } j(1-C\varepsilon^2j^{-2})|b_{j}|^2+r(\varepsilon) .
\end{split}
\end{equation}

Moreover, by  Appendix A in \cite{MN}, there exists a constant $c>0$, such that,
\be\label{a-}\sum_{k\geq 1}  k  |a_{k,-} e^{-c\theta_k} |^2 =\O( \varepsilon ^{-1/2} ),\ee
and thus, for $\varepsilon$ small enough,
\be
\label{Ak2}
\sum_{k\geq 2} k| A_{k,-}|^2=\sum_{k\geq 2} k| a_{k,-}e^{-c\theta_k}|^2e^{-2\theta_k (\frac{L}{\varepsilon}-c)} ={\mathcal O}(\varepsilon^{-1/2}e^{-2\pi L/\varepsilon}).
\ee

Therefore, we deduce from (\ref{est1})(with some new positive constants $C, \delta$), 
\begin{equation}
\begin{split}
\label{est1'}
 & \sum_{k\geq 1}(k-Ck^{-1}e^{-\delta/\varepsilon})|A_{k,+}|^2  \\
 & \leq (1+Ce^{-\delta /\varepsilon}) |A_{1,-}|^2-\frac {2\varepsilon}{\pi}(1+r_1(\varepsilon))\sum_{j\geq 1} \re\sqrt{\beta_j} |b_{j}|^2+ r_2(\varepsilon),
 \end{split}
\end{equation}
with
\be
\label{estr1}
r_1(\varepsilon) =\O\left(e^{-\delta /\varepsilon}\right)\quad ;\quad
r_2(\varepsilon) =\O\left(e^{-(\pi L +\delta )/\varepsilon} \right).
\ee

Now,  computing  $\la v (L_\varepsilon, \cdot),\varphi_1\ra_{L^2(|y|<\varepsilon)}$  and $\la \partial_xv (L_\varepsilon, \cdot), \psi_1\ra_{L^2(|y|<\varepsilon)}$ in two different ways (by using (\ref{match1})-(\ref{match2})), we find
\begin{eqnarray*}
&&\sum_{k\geq 1} \mu_k(A_{k,+}+A_{k,-}) =b_{1};\\
&& \frac1{\varepsilon}\theta_1(A_{1,+}-A_{1,-}) = -\sum_{j\geq 1}\nu_j(\sqrt{\beta_j }b_{j}-s_j'(L_\varepsilon)),
\end{eqnarray*}
with
$$
\mu_k:=\int_{-\varepsilon}^\varepsilon \psi_k(y)\varphi_1(y)dy =\left\{\begin{array}{l} 0 \mbox{ if $k$ is even};\\ \\
(-1)^{\frac{k-1}2}\frac{4k\sqrt{\varepsilon/\varepsilon_1}}{\pi (k^2-{(\varepsilon/{\varepsilon_1})}^2)} \cos{\frac \pi 2 \frac{\varepsilon}{\varepsilon_1} } \, \mbox{ if $k$ is odd},
\end{array}\right.
$$
and
$$
\nu_j:= \int_{-\varepsilon}^\varepsilon \varphi_j(y)\psi_1(y)dy =\left\{\begin{array}{l}
0 \mbox{ if $j$ is even};\\ \\
\frac{4\sqrt{\varepsilon/\varepsilon_1}\sin(((\varepsilon/\varepsilon_1) j -1)\pi/2)}{\pi ({(\varepsilon/\varepsilon_1)}^2j^2-1)} \mbox{ if $j\not= \frac{\varepsilon_1}{\varepsilon}$ is odd};\\ \\
\sqrt{(\varepsilon/\varepsilon_1)} \mbox{ if $j=\frac{\varepsilon_1}{\varepsilon}$ is odd}.
\end{array}\right.
$$
Using (\ref{Ak2}) again and (\ref{estbj1}), we obtain
 \begin{eqnarray}
\label{A+B}
&& |A_{1,+} + A_{1,-}| \leq  C e^{-(\pi  L+\delta)/2\varepsilon}+ \sum_{k\geq 2}|\frac{\mu_k }{\mu_1}A_{k,+}| +\frac{C}{\sqrt{\varepsilon}}e^{-\pi L/\varepsilon};\\
\label{A-B}
&& | A_{1,+} - A_{1,-}| \leq    \frac{\varepsilon}{|\theta_1|} \sum_{j\geq 1 }|\nu_j\sqrt{\beta_j}b_{j}|+C e^{-(\pi  L+\delta)/2\varepsilon},
\end{eqnarray}
with some new constant $C>0$.

Then, we observe that $|\mu_k/\mu_1|\leq (k-\frac{\varepsilon^2}{\varepsilon_1^2})^{-1}$ ($k$ odd), thus by (\ref{est1'}),
\begin{equation}\begin{split}
\sum_{k\geq 2}|\frac{\mu_k }{\mu_1}A_{k,+}|\leq \left(\sum_{k\geq 3}\frac1{k(k-\frac{\varepsilon^2}{\varepsilon_1^2})^2}\right)^{\frac12}\left(\sum_{k\geq 2}k|A_{k,+}|^2\right)^{\frac12}\\
\label{estsom1}
\leq \tau_1  \left(\alpha |A_{1,-}|^2- \beta\frac {2\varepsilon}{\pi}\sum_{j\geq 1}\re\sqrt{\beta_j}|b_{j}|^2+ r_2(\varepsilon) \right)^{\frac12} + Ce^{-(\pi L +\delta) /2\varepsilon},
\end{split}\end{equation}
where $\tau_1$ can be taken arbitrarily close to $(\sum_{k\geq 3}k^{-3})^{\frac12}<\frac12$, and $\alpha,\beta$ are positive numbers that tend to 1 as $\varepsilon \to 0$, and are such that $\alpha |A_{1,-}|^2- \beta\frac {2\varepsilon}{\pi}\sum_{j\geq 1}\re\sqrt{\beta_j}|b_{j}|^2+ r_2(\varepsilon)$ remains non negative for all $\varepsilon >0$ small enough.
Inserting (\ref{estsom1}) into (\ref{A+B}), we obtain
\begin{equation}
\label{A+B1}
|A_{1,+} + A_{1,-}| \leq \tau_1   \left(\alpha |A_{1,-}|^2-\beta\frac {2\varepsilon}{\pi}\sum_{j\geq 1}\re\sqrt{\beta_j}|b_{j}|^2+ r_2(\varepsilon)    \right)^{\frac12}  
+2Ce^{-(\pi L +\delta) /2\varepsilon}.
\end{equation}
On the other hand, going back to (\ref{A-B}), the Cauchy-Schwarz inequality gives,
\begin{eqnarray}
 \label{A-B2}
\frac{\varepsilon}{ |\theta_1|} \sum_{j\geq 1}|\nu_j\sqrt{\beta_j}b_{j}|
 \leq  \tau_2  \left(\frac{2\varepsilon}{\pi}\sum_{j\geq 1}|b_{j}|^2|\sqrt{\beta_j}|\right)^{\frac12}
 \end{eqnarray}
 with
\be
 \begin{split}
 \tau_2^2 &=  \frac {\varepsilon\pi}{2|\theta_1|^2}\sum_{j\geq 1}  j|\nu_{j} |^2|\sqrt{\beta_j}|\\
 &=\frac{16}{\pi^2}(1+\O(\varepsilon^2))\sum_{j\geq 1,\, j\,{\rm odd}}\frac{\varepsilon}{\varepsilon_1}\frac{\frac{j\varepsilon}{\varepsilon_1}\sin^2\left( (\frac{j\varepsilon}{\varepsilon_1}-1)\frac{\pi}2\right)}{\left( (\frac{j\varepsilon}{\varepsilon_1})^2-1\right)^2}(1+\O(j^{-2}))
 \end{split}
\ee
 In particular, when $\varepsilon \to 0$, then $\tau_2$ tends to 
 $\Gamma_2 :=  \frac{2\sqrt2}{\pi} \left( \int_0^\infty    \frac{ x \sin^2((x-1)\pi/2)}  {(x^2-1)^2}dx \right)^{\frac12}$,
and we deduce from (\ref{A-B}) and (\ref{A-B2}), plus the fact that $\im\sqrt{\beta_j} =\O (e^{-\delta/\varepsilon})$ uniformly,
\be
\label{A-B1}
 | A_{1,+} - A_{1,-}| \leq  \tilde\tau_2 \left(\frac{2\varepsilon}{\pi}\sum_{j\geq 1}\re\sqrt{\beta_j}|b_{j}|^2\right)^{\frac12}+ C e^{-(\pi L+\delta)/2\varepsilon},
\ee
where $\tilde\tau_2$ can be taken arbitrarily close to $\Gamma_2$. Actually, $\Gamma_2$ can be computed exactly, and one finds,
$$
\Gamma_2 = \frac{2\sqrt2}{\pi} \left( -\frac12+\frac{\pi}4{\rm Si}( \pi) \right)^{\frac12}\approx 0,879.
$$
(Here, ${\rm Si} (x):= \int_0^x\frac{\sin t}{t}dt $.)

Summing (\ref{A+B1}) with (\ref{A-B1}), and using the triangle inequality, we finally obtain
\be
\label{estA1}
2|A_{1,-}|  \leq \tau_1   \sqrt{\alpha |A_{1,-}|^2-\beta X+r_2(\varepsilon)}+\tau_2\sqrt{X}  +3Ce^{-(\pi L +\delta) /2\varepsilon},
\ee
where we have set
$$
X:=\frac {2\varepsilon}{\pi}\sum_{j}\re\sqrt{\beta_j}|b_{j}|^2.
$$
Now, an elementary computation shows that the map
$$
[0,A^2]\ni Y\mapsto \tau_1   \sqrt{A^2-\beta Y^2}+\tau_2 Y
$$
reaches its maximum at $Y=\frac{\tau_2^2}{\beta\tau_1^2+\tau_2^2}A/\sqrt\beta$, and the maximum value is
$$
\left(\sqrt{\tau_1^2+\beta^{-1}\tilde\tau_2^2}\right) A.
$$
Therefore, we deduce from (\ref{estA1}),
\be
\begin{split}
2|A_{1,-}|&\leq  \left(\sqrt{\tau_1^2+\beta^{-1}\tilde\tau_2^2}\right)\sqrt{\alpha |A_{1,-}|^2+ r_2(\varepsilon)} +3Ce^{-(\pi L +\delta) /2\varepsilon}\\
&\leq \left(\sqrt{\alpha(\tau_1^2+\beta^{-1}\tilde\tau_2^2)}\right) |A_{1,-}|+\O(e^{-(\pi L +\delta) /2\varepsilon}).
\end{split}
\ee
Since $\sqrt{\alpha(\tau_1^2+\beta^{-1}\tau_2^2)}$ tends to $\sqrt{\sum_{k\geq 2}k^{-3} + \Gamma_2^2}$ as $\varepsilon\to 0$, and
$$
\sum_{k\geq 3}k^{-3} + \Gamma_2^2 \leq \frac14+\frac8{10} <4,
$$
 we have proved,
\begin{proposition}\sl Under the assumption (\ref{absurd}),
there exist two constants $C, \delta >0$ such that, for any $\varepsilon >0$ small enough, one has,
\be
\label{atop}
|A_{1,-}|\leq  Ce^{-(\pi L +\delta) /2\varepsilon}.
\ee
\end{proposition}

\section{ End of the proof}

By Assumption {\bf (H)}, we see that the Dirichlet eigenfunction $u_0$ satisfies the hypothesis of \cite{BHM} Lemma 3.1. Then, following the arguments of \cite{BHM} leading to (13) in that paper, and using again \cite{HM}, Proposition 3.1 and Formula (5.13), we conclude that for any $\delta >0$ and any $x\in (0,L)$, there exists $C_1$ such that the resonant state $u_\varepsilon$ verifies (see \cite{BHM}, Formula (13)),
\be\label{utube}
\|u_\varepsilon\|_{L^2([x,L_\varepsilon]\times [-\varepsilon,\varepsilon] )}\geq \frac1{C_0} \varepsilon^{4.5+\delta} e^{-\pi x/2\varepsilon}.
\ee
Using this  estimate, we can now prove as in \cite{MN}, Proposition 8.2, the following proposition, that contradicts the inequality $(\ref{atop})$, and  thus completes  the proof the theorem \ref{mainth}.

\begin{proposition}\sl\label{step4} For any $\delta >0$, there exists $C>0$, such that
\begin{eqnarray}  
   |A_{1,-}| \geq   \frac1{C}  \varepsilon^{4.5+\delta}   e^{- \pi L/2\varepsilon},\end{eqnarray}
   for $\varepsilon >0$ small enough.
   \end{proposition}
   
    \begin{proof} Starting from 
(\ref{est1'}), we see,
\begin{eqnarray}
\label{estA+}
 && \sum_{k\geq 1} |A_{k,+}|^2   \leq  (1+Ce^{-\delta /\varepsilon}) |A_{1,-}|^2+ Ce^{-(\pi L +\delta) /\varepsilon}.
\end{eqnarray}
Then, computing the quantity $\|u_\varepsilon\|_{L^2([x,L]\times [-\varepsilon,\varepsilon] )}$ 
by using the expression (\ref{letube}), we obtain (see \cite{MN}, proof of Proposition 8.2),
\be
\begin{split}
 \|u_\varepsilon\|^2_{L^2([x,L_\varepsilon]\times [-\varepsilon,\varepsilon] )}\leq 
4\sum_{k\geq 1} |A_{k,+}|^2 
+4\sum_{k> 1} |a_{k,-}|^2 
 e^{-2x\re{\theta_k}/\varepsilon } \\
 +\varepsilon |a_{1,-}|^2  
 e^{-2x\re{\theta_1}/\varepsilon }.
 \end{split}
\ee

Using  (\ref{estA+})  and (\ref{a-}), we deduce
\be
\begin{split}
 \label{step1}
 \|u_\varepsilon\|^2_{L^2([x,L_\varepsilon]\times [-\varepsilon,\varepsilon] )}\leq C \varepsilon  |a_{1,-}|^2  
 e^{-2x\re{\theta_1}/  \varepsilon} +  C |a_{1,-}|^2  
 e^{-2L\re{\theta_1}/  \varepsilon}\\
+ C \varepsilon ^{-C} e^{-2x\re{\theta_1}/\varepsilon} e^{-2C_0x/\varepsilon}+C e^{-(\pi  L+\delta)/\varepsilon}, 
 \end{split}
 \ee
and thus, using (\ref{utube}), we finally obtain,
\begin{eqnarray} \label{step3} 
 \varepsilon^{9+2\delta}   \leq  C  |a_{1,-}|^2,  
\end{eqnarray}
and the result is proved.

\end{proof}

\section{An extension to larger dimensions}\label{dim3}

Here, we consider the similar problem in dimension $n\geq 3$, obtained by taking tubes with square sections. That is,  $\mathcal C$ is a regular bounded open subset   of $\R^n$, and we have (in Euclidean coordinates $x=(x_1,\dots, x_n)= (x_1,x')\in\R\times\R^{n-1}$),
\be
\label{hypgeom3}
 \begin{array}{l}
\overline{\mathcal C}\subset \mathcal B;\\
(0,0)\in \partial{\mathcal C} ; (L,0)\in \partial{\mathcal B};\\
{[0,L]}\times\{0\}\subset \overline{\mathcal B}\backslash {\mathcal C};\\
\mbox{Near } M_0:=(L,0),\, {\mathcal B} \mbox{ is convex and } \partial{\mathcal B}
\mbox{ is symmetric with}\\ \mbox{respect to } \{x_j=0\} \mbox{ for all } j\geq 2.
\end{array}
\ee
\begin{remark}\sl In particular, this also contains the case where $\partial{\mathcal B}$ is flat near $M_0$, that is when $\{ (L, x_2,\dots,x_n)\, ;\, |x_j| < \varepsilon_0,\, j=2,\dots,n\}\subset \partial{\mathcal B}$ for some $\varepsilon_0>0$.
\end{remark}
Then, setting $Q_\varepsilon:=\{ (x_2,\dots,x_n)\, ;\, |x_j| < \varepsilon,\, j=2,\dots,n\}$, ${\mathcal T}(\varepsilon):= [-\varepsilon_0, L]\times Q_\varepsilon \cap (\R^n\backslash {\mathcal C})$, and ${\mathbf E}:= \R^n \backslash \overline{\mathcal B}$, we consider the resonances of the resonator $\Omega(\varepsilon):={\mathcal C}\cup{\mathcal T}(\varepsilon)\cup{\mathbf E}$. 

As before, let $\lambda_0$ be an eigenvalue of $-\Delta_{\mathcal{C}}$, and let  $u_0$ be the corresponding normalized eigenfunction. 

In this situation, the lower estimate of \cite{HM} (see also \cite{BHM}) becomes
$$
\im \rho (\varepsilon) ={\mathcal O}(e^{-(1-\delta)\pi L\sqrt{n-1} /\varepsilon}),
$$
where $\rho (\varepsilon)$ stands for any resonance that tends to $\lambda_0$ as $\varepsilon\to 0_+$, and $\delta >0$ is arbitrary.

We assume again,

Assumption ({\bf H}):
\begin{itemize}
\item[] $\lambda_0$ is simple;
\item[] $u_0$ does not vanish on ${\mathcal C}$ near the point $(0,0)$.
\end{itemize}
Then, we have
 \begin{theorem}\sl 
 \label{mainth'}
Under Assume {\bf(H)} and $2\leq n\leq 12$. Then, for any $\delta >0$ there exists $C_\delta >0$ such that, the only resonance $\rho(\varepsilon)$ close to $\lambda_0$ satisfies,
 $$
 |\im \rho(\varepsilon)|\geq  \frac1{C_\delta}e^{-\pi (1+\delta) L\sqrt{n-1}/\varepsilon},
 $$
 uniformly as $\varepsilon\to 0_+$.
 \end{theorem}

 \begin{proof} 
 The computations are very similar to those in dimension 2, and we highlight here only what is specific to dimension n. 
 The notations are similar, but their meaning is modified as follows. 
 For $k=(k_2,\dots, k_n)\in\N^{n-1}$ (where $\N:= \{ 1,2,3,\dots\}$), we set
\begin{eqnarray*}
&& \alpha_k := \left(\frac{k_2\pi}2,\dots,\frac{k_n\pi}2\right)\in \R^{n-1};\\
&& \theta_k:= \sqrt{|\alpha_k|^2-\varepsilon^2\rho(\varepsilon)};\\
&& \beta_k:= |\alpha_k|^2\varepsilon_1^{-2}-\rho(\varepsilon);\\
&& \psi_k(x'):= \psi_{k_2}(x_2)\dots\psi_{k_n}(x_n);\\
&& \varphi_k(x'):= \varphi_{k_2}\dots(x_2)\varphi_{k_n}(x_n).
\end{eqnarray*}

(Here, $|k|$ stands for the Euclidean norm of $k$ in $\R^{n-1}$.) With these notations, the  formulas (\ref{letube})-(\ref{match2}) remain valid with the following changes:

\begin{itemize}
\item $\sum_{k=1}^\infty$ must be replaced by $\sum_{k\in\N^{n-1}}$ , and analog for $\sum_{j=1}^\infty$;
\item $y$ must be replaced by $x'$;
\item $(-\varepsilon, \varepsilon)$ and $(-\varepsilon_1,\varepsilon_1)$ must be respectively replaced by $Q_\varepsilon$ and $Q_{\varepsilon_1}$ (where $\varepsilon_1$ is taken such that $\frac{(n-1)\pi^2}{4\varepsilon_1^2} >\lambda_0$).
\end{itemize}

Computing in two ways the quantities $\la v, \partial_x v\ra_{\{L\}\times Q_\varepsilon}$, $\la v,\varphi_{1,\dots,1}\ra_{\{L\}\times Q_\varepsilon}$, and $ \la\partial_x v,\psi_{1,\dots,1}\ra_{\{L\}\times Q_\varepsilon}$, we find the following analogs of (\ref{est1'})-(\ref{A-B}):
\begin{eqnarray*}
 && \sum_{k\in \N^{n-1}}(|k|-C|k|^{-1}e^{-\delta/\varepsilon})|A_{k,+}|^2  \\
 && \hskip 0.2cm\leq  (1+Ce^{-\delta /\varepsilon}) |A_{1,\dots,1,-}|^2-\frac{2\varepsilon}{\pi}(1+r_1)\sum_{j\in\N^{n-1}} \re\sqrt{\beta_j}|b_{j}|^2+r_2 ;\\
 && |A_{1,\dots,1,+} + A_{1,\dots,1,-}| \leq Ce^{-(\pi L\sqrt{n-1}+\delta)/2\varepsilon} + \sum_{|k|>\sqrt{n-1}}|\frac{\mu_k }{\mu_{1,1}}A_{k,+}|\\
 &&\hskip 7cm +\frac{C}{\sqrt{\varepsilon}}e^{-\pi L\sqrt{4+(n-2)^2}/2\varepsilon};\\
 && | A_{1,1,+} - A_{1,1,-}| \leq  \frac{\varepsilon}{|\theta_{1,\dots,1}|} \sum_{j\in\N^{n-1}}|\nu_j\sqrt{\beta_j}b_{j}| +Ce^{-(\pi L\sqrt{n-1}+\delta)/2\varepsilon},
\end{eqnarray*}
where we have set
$$
 \nu_{j}:= \nu_{j_2}\dots \nu_{j_n}\quad ;\quad\mu_{k}:=\mu_{k_2}\dots\mu_{k_n},
$$
and with,
$$
r_1 =\O (e^{-\delta /\varepsilon})\quad ; \quad r_2=\O(e^{-(\pi L\sqrt{n-1} +\delta)/\varepsilon}).
$$
Using the fact that $\mu_{k_2,\dots,k_n}/\mu_{1,\dots,1}\leq (k_2-\frac{\varepsilon^2}{\varepsilon_1^2})^{-1}\dots(k_n-\frac{\varepsilon^2}{\varepsilon_1^2})^{-1}$ ($k_2,\dots,k_n$ odd), this also gives
\begin{eqnarray}
&& |A_{1,\dots,1,+} + A_{1,\dots,1,-}|\nonumber\\
 \label{A+-A-}
&& \hskip 1cm  \leq \tilde\tau_1\left( |A_{1,\dots,1,-}|^2 -\frac{\varepsilon}{\varepsilon_1}\sum_{j\in\N^{n-1}}|j||b_{j}|^2 \right)^{\frac12} +Ce^{-(\pi L\sqrt2+\delta)/2\varepsilon} ,
\end{eqnarray}
where $\tilde\tau_1$ can be taken arbitrarily close to 
\be
\label{tautilde1}
J_1:=(\sum_{|k|^2 > n-1\,;\, k_j\,{\rm odd}}|k|^{-1}k_2^{-2}\dots k_n^{-2})^{\frac12} 
=(\sum_{k_j\,{\rm odd}}|k|^{-1}k_2^{-2}\dots k_n^{-2} -\frac1{\sqrt{n-1}})^{\frac12}.
\ee
A rough estimate on $J_1$ can be obtained by writing,
$$
J_1^2\leq \frac1{\sqrt{n-1}}\left(( \sum_{\ell \in\N \,{\rm odd}}\frac1{\ell^2})^{n-1}-1\right)\leq  \frac1{\sqrt{n-1}}\left( \left(\frac{\pi^2}{8}\right)^{n-1}-1\right).
$$ 

In a similar way we obtain,
\be
 \label{A++A-}
| A_{1,\dots,1,+} - A_{1,\dots,1,-}| \leq   \tilde\tau_2\left(\frac{\varepsilon}{\varepsilon_1}\sum_{j\in\N^{n-1}}|j||b_{j}|^2 \right)^{\frac12} +Ce^{-\pi L(\sqrt{n-1} +\delta )/2\varepsilon},
\ee
where $\tilde\tau_2$ can be taken arbitrarily close to the quantity
 \be
 \label{tautilde2}
 J_2= \frac{4^{n-1}}{(\pi\sqrt 2)^{n-1}\sqrt{n-1}}\left(\int_{\R_+^{n-1}}\frac{|x|\sin^2((x_1-1)\pi/2)\dots\sin^2((x_{n-1}-1)\pi/2)}{(x_1^2-1)^2\dots(x_{n-1}^2-1)^2}dx_1\dots dx_{n-1}\right)^{\frac12}.
\ee
Writing $|x|\leq |x_1|+\dots +|x_{n-1}|$ and making permutations on the variables, we obtain,
$$
J_2\leq  \frac{4^{n-1}}{(\pi\sqrt 2)^{n-1}}\left( \int_0^{+\infty}\frac{t\sin^2((t-1)\pi/2)}{(t^2-1)^2}dt\right)^{\frac12}\left( \int_0^{+\infty}\frac{\sin^2((t-1)\pi/2)}{(t^2-1)^2}dt\right)^{\frac{n-2}2}
$$
Setting
$$
L_1:=  \int_0^{+\infty}\frac{t\sin^2((t-1)\pi/2)}{(t^2-1)^2}dt\quad ;\quad L_2:= \int_0^{+\infty}\frac{\sin^2((t-1)\pi/2)}{(t^2-1)^2}dt,
$$
it becomes,
$$
J_2\leq \left(\frac{L_1}{L_2}\right)^\frac12 \left( \frac{4\sqrt{L_2}}{\pi\sqrt 2}\right)^{n-1}.
$$
The integrals $L_1$ and $L_2$ can be computed exactly, and one finds,
 $$
L_1 = -\frac12 +\frac{\pi}4 {\rm Si}(\pi)\approx 0.9545\quad ; \quad L_2 =\frac{\pi^2}8.
 $$
 In particular, for $\varepsilon$ small enough, we have
 \be
 \label{tautilde}
 \tilde\tau_1^2 + \tilde\tau_2^2 <  \frac8{10}+\frac1{\sqrt{n-1}}\left( \left(\frac{\pi^2}{8}\right)^{n-1}-1\right),
\ee
and one can check that this quantity is strictly less than 4 when $2\leq n\leq 12$.

At this point, we can complete the proof as in the 2 dimensional case.\end{proof}

\section{Appendix}
We prove Lemma \ref{lemmeHankel}. For $k\geq 0$, we can represent $h_k(R)=H_k(R\sqrt\rho)$ by the formula (see, e.g., \cite{Wa}),
$$
h_k(R)=\frac1{i\pi}\int_{-\infty}^{+\infty +i\pi}e^{R\sqrt\rho\sinh t-kt}dt,
$$
that we split into,
$$
\begin{aligned}
h_k(R)&=\frac1{i\pi}\int_{-\infty}^0e^{R\sqrt\rho\sinh t-kt}dt+\frac1{\pi}\int_0^\pi e^{i(R\sqrt\rho\sin \theta-k\theta)}d\theta\\
& \hskip 5cm +\frac1{i\pi}\int_0^{+\infty}e^{-R\sqrt\rho\sinh t-kt-ik\pi}dt\\
&=\frac1{i\pi}\int_{-\infty}^0e^{R\sqrt\rho\sinh t-kt}dt+\O(1).
\end{aligned}
$$
In the latter integral, we make the change of variable: $t\mapsto -t-\ln k$, and we obtain,
$$
h_k(R)=\frac{k^k}{i\pi}\int_{-\ln k}^{+\infty} e^{k\psi (t)}a_k(t)dt +\O(1),
$$
with,
$$
\psi (t):=t-R\sqrt\rho e^t/2\quad ;\quad a _{k}(t):=e^{R\sqrt\rho e^{-t}/2k}.
$$
Here, we observe that, for any $j\geq 0$, we have $a_{k,R}^{(j)}(t)=\O(1)$ uniformly on $[-\ln k, +\infty)$. Moreover, the phase function $\psi$ admits a unique critical point at $t_c:=\ln (2/R\sqrt\rho)$, and $ \psi''(t_c)= -1$. In particular, since also $\re t_c>0$ and  $\im t_c\to 0$ as $\varepsilon \to 0$,  we can apply the method of steepest descent in order to estimate this integral, and we obtain,
$$
h_k(R)=-i\sqrt{\frac{2}{\pi}}\,k^{k-\frac12} e^{k\psi (t_c)}\left(a_k(t_c) +\O(k^{-1})\right)+\O(1),
$$
that is,
$$
h_k(R)=-i\sqrt{\frac{2}{\pi}}\,k^{k-\frac12}\left(\frac2{eR\sqrt\rho}\right)^k\left(a_k(t_c) +\O(k^{-1})\right)+\O(1).
$$
Since $a_k(t_c) =1+\O(k^{-1})$, the result follows.


\begin{thebibliography}{10}
 
 \bibitem[A]{A} 
 R.~ A.~Adams.
 \newblock{\em Sobolev Spaces.}
 \newblock Academic Press, Boston,1975.

 \bibitem[Br]{Br} 
H. Brezis.
 \newblock{\em Functional Analysis,
Sobolev Spaces and Partial
Differential Equations.} 
 \newblock  Springer Universitext (ISBN 978-0-387-70913-0), 2011.

\bibitem[BHM]{BHM}
R.M~Brown, P.D.~Hislop and A.~Martinez.
\newblock {L}ower {B}ounds on {E}igenfunctions and the first {E}igenvalue {G}ap.
\newblock {\em Differential Equations with Applications to Mathematical Physics}
\newblock  W.F.Ames, E.M.Harell, J.V.Herod (Ed.), Mathematical and Science in Engineering, Vol. 192, Academic Press 1993.
 22: 269--279, 1971.
 
  \bibitem[Bu1]{Bu1} 
 N.~Burq.
 \newblock  D\'ecroissance de l'\'energie locale de l'\'equation des ondes pour le probl\`eme ext\'erieur et absence de r\'esonance au voiosinage du r\'eel, 
 \newblock {\em Acta Math. }, 180, 1-29, 1998.
 
 \bibitem[Bu2]{Bu2} 
 N.~Burq.
 \newblock  Lower bounds for shape resonances widths of long range Schr\"odinger operators, 
 \newblock {\em Am. J. Math. }, 124, 2002.
 
  \bibitem[CP]{CP} 
  J.~Chazarain, A.~Piriou.
  \newblock {\em Introduction \`a la Th\'eorie des \'Equations aux D\'eriv\'ees Partielles Lin\'eaires.} \newblock Gauthier-Villars, 1981.
 
  \bibitem[FLM]{FLM} 
  S.~Fujiie, A.~Lahamar-Benbernou, A.~ Martinez.
  \newblock { Width of shape resonances for non globally analytic potentials.}
  \newblock{\em J. Math. Soc. Japan Volume} 63, Number 1, 1-78, 2011. 
  
 \bibitem[FL]{FL} 
 C.~Fernandez, R.~Lavine.
 \newblock { Lower bounds for resonance width in potential and obstacle scattering.}
 \newblock{\em Comm. Math. Phys.,} 128, 263-284,1990.
 
 \bibitem[Ha]{Ha} 
 E.~Harrel.
 \newblock {General lower bounds for resonances in one dimension.}
 \newblock{ \em Comm. Math. Phys.} 86, 221-225,1982.
 
  \bibitem[HaSi]{HaSi} 
 E.~Harrel, B.~Simon.
 \newblock {The mathematical theory of resonances whose widths are exponentially small.}
 \newblock{ \em Duke Math.} 47 n.4, 845-902,1980.
 
 \bibitem[HS]{HS} 
 B.~Helffer, J.~ Sj\"ostrand.
 \newblock { R\'esonances en limite semiclassique.}
 \newblock{\em  Bull. Soc. Math. France, M\'emoire }{\bf 24/25}, 1986.

\bibitem[HeM]{HeM}
B.~Helffer, A.~Martinez. 
\newblock {\sl Comparaison entre les diverses notions de r\'esonances.}
\newblock{ \em Helv. Phys. Acta,} Vol.60, p.992-1003, 1987.

\bibitem[HM]{HM}
P. D.~Hislop and A.~Martinez. 
\newblock Scattering resonances of a Helmholtz resonator.
\newblock{\em  Indiana Univ. Math. J. }40 no. 2, 767-788, 1991.

\bibitem[LL]{LL}
J. Le Rousseau and G. Lebeau. 
\newblock On Carleman estimates for elliptic and parabolic operators. Applications to unique continuation and control of parabolic equations.
\newblock{\em    ESAIM: Control, Optimisation and Calculus of Variations},  Volume 18 / Issue 03 / July 2012, pp 712-747

\bibitem[LR]{LR}
G. Lebeau and L. Robbiano. 
\newblock Contr\^ole exact de l'\' equation de la chaleur.
\newblock{\em    Comm. Part. Diff. Eq.}, 20 (1\&2), 335-356, 1995.

  \bibitem[Ma]{Ma} 
A.~Martinez.
 \newblock {An Introduction to Semiclassical and Microlocal Analysis.}
 \newblock{\em Springer-Verlag New-York}, UTX Series, ISBN: 0-387-95344-2, 2002.

  \bibitem[MN]{MN}
 A.~Martinez and L.~Nedelec
\newblock Optimal lower bound of the resonance widths for a Helmholtz tube-shaped resonator
\newblock{\em   J. Spectr. Theory },2, 203-223, 2012.

\bibitem[Sj]{Sj}  
J.~Sj\"ostrand.
\newblock Lecture on resonances.
\newblock {\em Preprint}, 2002.

\bibitem[SZ]{SZ}  
J.~Sj\"ostrand and M.~Zworski.
\newblock Complex scaling and the distribution of scattering poles.
\newblock {\em J. Amer. Math. Soc.}, 4:729--769, 1991.

\bibitem[Wa]{Wa} Watson, G.N. 
\newblock A Treatise on the Theory of Bessel Functions, $2^d$ edition. Cambridge
Univ. Press, Cambridge, 1944.
\end{thebibliography}
\end{document}